\documentclass[10pt]{article}
\usepackage[latin5]{inputenc}
\usepackage{amsfonts}
\usepackage{amssymb}
\usepackage{amsmath}
\usepackage{cite}
\usepackage{bbm}
\usepackage{graphicx}
\usepackage{pstricks,pst-math,pst-infixplot}
\usepackage[dotinlabels]{titletoc}
\usepackage{psfrag}
\usepackage[shortlabels]{enumitem}
\usepackage{maplestd2e}
\usepackage{lineno}

\newcounter{theorem}
\newtheorem{theorem}{Theorem}

\newtheorem{example}{Example}

\newenvironment{proof}[1][Proof]{\textbf{#1.} }{\rule{0.5em}{0.5em}}
\begin{document}

\title{AN ALGORITHM FOR THE CONSTRUCTION OF THE TIGHT SPAN OF FINITE SUBSETS OF THE MANHATTAN PLANE}
\author{Mehmet Kılıç \footnote{Anadolu University, Department of Mathematics, 26470, Eskişehir, Turkey, kompaktuzay@gmail.com},\, Şahin Koçak \footnote{Anadolu University, Department of Mathematics, 26470, Eskişehir, Turkey, skocak@anadolu.edu.tr},\, and Yunus Özdemir \footnote{Anadolu University, Department of Mathematics, 26470, Eskişehir, Turkey, yunuso@anadolu.edu.tr}}
\date{ }

\maketitle

\begin{abstract}
We give a simple algorithm for the construction of  the tight span of a finite subset of the
Manhattan plane.
\end{abstract}

\section{Introduction}

The notion of tight span, which plays an important role in phylogenetic analysis, was invented by Dress~\cite{dre}. One of his motivations was to construct optimal realizations of finite metric spaces in metric graphs (i.e. isometric embeddings into metric graphs with minimal total weight). It turned out that the tight span was actually a rediscovering of the notion of injective hull due to Isbell~\cite{isb}.

Tight span of subsets of the Manhattan plane is an especially interesting and important special case, which is investigated in the elaborate work of Eppstein~\cite{epp} (We note however that the important Lemma~9 in \cite{epp} requires a slight correction. Eppstein remarks that it can be fixed by taking the closure of his so called orthogonal hull \cite{epp1}).

In a previous work we constructed the tight span of a path connected subset of the Manhattan plane by a very simple device of hatching the subset successively in both axis-directions. In the present note we give a simple algorithm for the construction of the tight span of a finite subset of the Manhattan plane by embedding the given finite subset into a suitable path connected subset and then applying the hatching operation.

In Section~\ref{Section1} we formulate the algorithm and prove that the
resulting set is indeed the tight span of the given finite set
(Theorem~\ref{sn2}). In Section~\ref{Section2} we give some examples and in the
appendix, we give the source-code of the program, whose input is a
finite subset of the Manhattan plane and whose output is its tight
span.
\section{An Algorithm for the Construction of the Tight Span of Finite Subsets}\label{Section1}

We recall that the Manhattan plane is the metric space $\mathbb{R}_1^2=(\mathbb{R}^2,d_1)$ with \[d_1((x_1,y_1),(x_2,y_2))=|x_1-x_2|+|y_1-y_2|.\]
A geodesic between two points is a shortest path between these points parameterized by arc-length and a shortest path between two points is a path with length realizing the distance between these points (see \cite{bur},\cite{pap}). In the Manhattan plane a shortest path between two points $(x_1,y_1)$ and $(x_2,y_2)$ with $x_1\leq x_2$ and $y_1\leq y_2$ is a path $\alpha=(\alpha_1,\alpha_2):[a,b]\rightarrow\mathbb{R}^2$ with $\alpha(a)=(x_1,y_1)$, $\alpha(b)=(x_2,y_2)$ and $\alpha_1(t_1)\leq\alpha_1(t_2)$ and $\alpha_2(t_1)\leq\alpha_2(t_2)$ for $t_1\leq t_2$ (see Figure~\ref{somegeodesics}a for a sample of geodesics). Likewise, for two points $(x_1,y_1)$ and $(x_2,y_2)$ with $x_1\leq x_2$ and $y_1\geq y_2$ the component $\alpha_1$ of $\alpha$ must be non-decreasing and the component $\alpha_2$ must be non-increasing c.f.~\cite{pap}; see Figure~\ref{somegeodesics}b. If $x_1=x_2$ or $y_1=y_2$ for a pair of points $(x_1,y_1)$ and $(x_2,y_2)$, then there exist a unique geodesic between these points which is the linear segment connecting them (see Figure~\ref{somegeodesics}c).

\begin{figure}[h]
\begin{center}
\begin{pspicture}(-5.25,-6.75)(7.5,-2)
\psline[linewidth=0.25pt]{->}(-5.25,-4.5)(-1.75,-4.5)
\psline[linewidth=0.25pt]{->}(-3.5,-6.25)(-3.5,-2.5)
\uput[u](2.5,-4.5){$x$} \uput[u](0.75,-2.5){$y$}\uput[u](7,-4.5){$x$}
\uput[u](-1.75,-4.5){$x$} \uput[u](-3.5,-2.5){$y$}
\psline[linewidth=0.25pt]{->}(-1,-4.5)(2.5,-4.5)
\psline[linewidth=0.25pt]{->}(0.75,-6.25)(0.75,-2.5)
\uput[u](5,-2.5){$y$}
\psline[linewidth=0.25pt]{->}(3.25,-4.5)(7,-4.5)
\psline[linewidth=0.25pt]{->}(5,-6.25)(5,-2.5)

\psdots(-5,-5.5)(-2,-3)
\uput[d](-4.75,-5.5){$(x_1,y_1)$}
\uput[u](-2,-3){$(x_2,y_2)$}
\psline[linestyle=dotted](-5,-5.5)(-2,-5.5)(-2,-3)
\psline[linestyle=dashed](-2,-3)(-5,-3)(-5,-5.5)
\pscurve[linewidth=0.25pt](-5,-5.5)(-3,-5)(-2,-3)
\psline[linewidth=1.25pt](-5,-5.5)(-4.5,-5)(-4,-5)(-4,-4)(-3,-4)(-2,-3)
\psdots(-0.75,-3)(2.25,-5.5)
\uput[u](-0.5,-3){$(x_1,y_1)$}
\uput[d](2.25,-5.5){$(x_2,y_2)$}

\psline[linestyle=dashed](2.25,-5.5)(2.25,-3)(-0.75,-3)
\psline[linestyle=dotted](-0.75,-3)(-0.75,-5.5)(2.25,-5.5)
\pscurve[linewidth=0.25pt](-0.75,-3)(0,-3.75)(1,-5)(2.25,-5.5)
\psline[linewidth=1.25pt](-0.75,-3)(0,-3.5)(1,-3.5)(1,-4.5)
\pscurve(1,-4.5)(2,-5)(2.25,-5.5)
\psdots(3.5,-3)(3.5,-4.25)
\psline(3.5,-3)(3.5,-4.25)
\psdots(6.5,-5)(4,-5)
\psline(6.5,-5)(4,-5)

\uput[d](-3.5,-6.5){(a)}
\uput[d](0.75,-6.5){(b)}
\uput[d](5,-6.5){(c)}
\end{pspicture}
\caption{Geodesics in the Manhattan plane.}
\label{somegeodesics}
\end{center}
\end{figure}
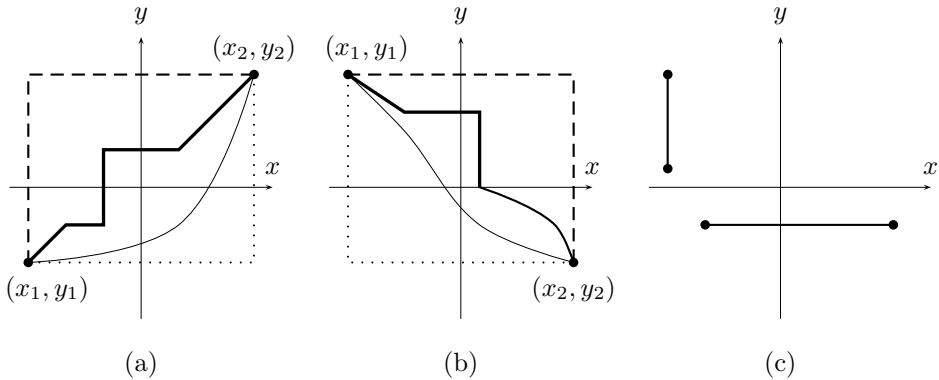

We will also need the hatching operations for subsets of the Manhattan plane. For a subset $A\subseteq\mathbb{R}_1^2$, let $A_{x_0}$ denote the intersection of the set $A$ with the vertical line $x=x_0$ and $A^{y_0}$ the intersection of $A$ with the horizontal line $y=y_0$.
Given any subset $X\subseteq \mathbb{R}_1^2$ lying on a horizontal or vertical line in $\mathbb{R}_1^2$, denote the
minimal segment (possibly infinite or empty) containing the set $X$ and contained in the same horizontal or vertical line by $[X]$.
We can now define
\[
L_x(A)=\bigcup_{y\in\mathbb{R}}[A^y]
\]
and
\[
L_y(A)=\bigcup_{x\in\mathbb{R}}[A_x].
\]
We call the set $L_x(A)$ the hatching of $A$ in the $x$-direction and $L_y(A)$ the hatching of $A$ in the $y$-direction. For a path connected subset $A\subseteq\mathbb{R}_1^2$, the operations $L_x$ and $L_y$ commute and we define $L(A)=L_x(L_y(A))=L_y(L_x(A))$ (see \cite{kilicarxiv}).

For  later use we want to fix the following notation:

For $p=(x_1,y_1)\in\mathbb{R}_1^2$ and $\varepsilon_1,\varepsilon_2=\pm$, we call the set
\[
Q_p^{\varepsilon_1\varepsilon_2}=\{q=(x_2,y_2)\in \mathbb{R}_1^2\ |\ \varepsilon_1(x_2-x_1)\geq 0 \mbox{ and } \varepsilon_2(y_2-y_1)\geq 0 \}
\]
the $\varepsilon_1\varepsilon_2-$quadrant of $p$ (see Figure~\ref{quadrants}).

Finally we want to remind the notion of tight span for a metric space.  Let $(X,d)$ be any metric space and consider the set of pointwise minimal functions $f:X\rightarrow \mathbb{R}^{\geq0}$ satisfying the property: \[f(p)+f(q)\geq d(p,q)\] for all $p,q\in X$. The tight span $T(X)$ of $X$ is then this set of functions with the supremum metric: \[d_{\infty}(f,g)=\sup_{p\in X}|f(p)-g(p)|.\]
The tight span of a metric space can be shown to be strictly intrinsic, i.e. between any two points there exists a geodesic (\cite{kilicarxivilk}).

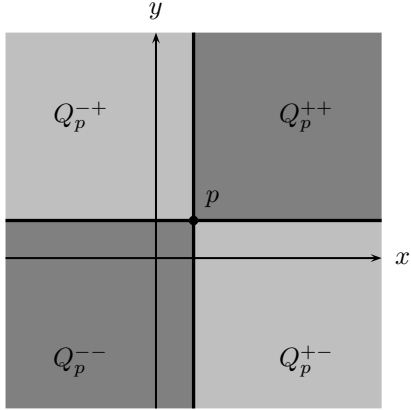
\begin{figure}[h!]
\begin{center}
\begin{pspicture}(-2,-2)(3.25,3.25)
\pspolygon*[linecolor=gray](0.5,0.5)(0.5,3)(3,3)(3,0.5)
\pspolygon*[linecolor=lightgray](0.5,0.5)(0.5,3)(-2,3)(-2,0.5)
\pspolygon*[linecolor=lightgray](0.5,0.5)(0.5,-2)(3,-2)(3,0.5)
\pspolygon*[linecolor=gray](0.5,0.5)(0.5,-2)(-2,-2)(-2,0.5)
\uput[u](0,3){$y$} \psline[linewidth=0.75pt]{->}(0,-2)(0,3)
\psline[linewidth=0.75pt]{->}(-2,0)(3,0) \uput[r](3,0){$x$}
\psdot(0.5,0.5)
\psline[linewidth=1.25pt](0.5,-2)(0.5,3)
\psline[linewidth=1.25pt](-2,0.5)(3,0.5)
\uput[u](0.75,0.5){$p$}
\uput[u](2,1.5){$Q_p^{++}$}
\uput[u](-1,1.5){$Q_p^{-+}$}
\uput[u](2,-1.75){$Q_p^{+-}$}
\uput[u](-1,-1.75){$Q_p^{--}$}
\end{pspicture}
\caption{Quadrants of the point $p$.} \label{quadrants}
\end{center}
\end{figure}

We have given in \cite{kilicarxiv} the following construction for the tight span of a path connected subset $A\subseteq\mathbb{R}_1^2$:

\begin{theorem}
For a path connected subset $A\subseteq\mathbb{R}_1^2$, $\overline{L(A)}$ is isometric to the tight span of $A$ (i.e. one can obtain the tight span by hatching the given subset successively in both axis-directions and taking the closure of the resulting set).

We can now formulate our main result:
\end{theorem}

\begin{theorem}
\label{sn2} Let $A\subseteq\mathbb{R}_1^2$ be a nonempty, finite
subset. Then, the tight span of $A$ can be constructed by the
following procedure:
\setlist[enumerate,1]{leftmargin=2cm}
 \begin{enumerate}[(1)]
\item[Step 1:] Choose a point $(a,b)\in A$ with the smallest ordinate (if there are several such points, choose the one with the smallest abscissa). Define ${\rm SPINE}=\{(a,b)\}$.

\item[Step 2:] If there are no points with higher ordinate than $b$, then go to Step~5.

\item[Step 3:] \mbox{ }
\setlist[enumerate,2]{leftmargin=0cm}
\begin{enumerate}[(2)]
\item[i)]  If there is a point $(x_1,y_1)\in A$ with $y_1>b$ and $x_1\geq a$; and there is a point $(x_2,y_2)\in A$ with $y_2>b$ and $x_2\leq a$, then let
\begin{eqnarray*}
t=\min\{\max\{y-b\, |\, (x,y)\in A,\ y>b,\ x\geq a\},\\ \max\{y-b\, |\,
(x,y)\in A,\ y>b,\ x\leq a\}\},
\end{eqnarray*}
and define ${\rm SPINE}={\rm SPINE} \, \cup \, [(a,b),(a,b+t)]$, where \mbox{$[(a,b),(a,b+t)]$} denotes the segment between the points $(a,b)$ and $(a,b+t)$, and set $b:=b+t$.
\item[ii)] If there is a point $(x_1,y_1)\in A$ with $y_1>b$ and $x_1\geq a$; but there is no point $(x,y)\in A$ with $y>b$ and $x\leq a$, then let
\begin{eqnarray*}
t=\min\{x-a\,|\, (x,y)\in A,\ y>b,\ x\geq a\},
\end{eqnarray*}
and define ${\rm SPINE}={\rm SPINE}\cup[(a,b),(a+t,b)]$, and set $a:=a+t$.
\item[iii)]  If there is a point $(x_1,y_1)\in A$ with $y_1>b$ and $x_1\leq a$; but there is no point $(x,y)\in A$ with $y>b$ and $x\geq a$, then let
\begin{eqnarray*}
t=\min\{a-x\,|\, (x,y)\in A,\ y>b,\ x\leq a\},
\end{eqnarray*}
and define ${\rm SPINE}={\rm SPINE}\cup[(a-t,b),(a,b)]$, and set $a:=a-t$.
\end{enumerate}

\item[Step 4:] Go to Step 2.
\item[Step 5:] Connect the points of $A$ with horizontal segments to the ${\rm SPINE}$. We call the resulting set the ${\rm SKELETON}$.
\item[Step 6:] Apply the vertical hatching operation to the ${\rm SKELETON}$. The resulting set will be isometric to the tight span $T(A)$ of the given \mbox{set $A$.}
\end{enumerate}
\end{theorem}
\begin{proof}
Let $p_1$ be the first point constructed in Step 3 (it may be $p_1=(a,b+t)$, $p_1=(a+t,b)$ or $p_1=(a-t,b)$ for some $t>0$). We will show that the subspace $A\cup\{p_1\}\subseteq\mathbb{R}_1^2$ can be isometrically embedded into $T(A)$. To see this, it will be enough to show that the function
\[
f_{p_1}:A\rightarrow\mathbb{R}^{\geq0},\, f_{p_1}(q)=d_1(p_1,q)
\]
belongs to $T(A)$, i.e.,
\begin{itemize}
\item
for every $q,r\in A$ $f_{p_1}(q)+f_{p_1}(r)\geq d_1(q,r)$ and
\item
for every $u\in A$, there must exist a $v\in A$ such that $f_{p_1}(u)+f_{p_1}(v)=d_1(u,v)$.
\end{itemize}
The first condition is obviously true by the triangle inequality and we will check the second condition by considering the positions of $p_1$ separately.
\begin{itemize}
\item[i)]
Let $p_1=(a,b+t)$. In this case every quadrant of the point $p_1$ contains a point of $A$. Given any $u\in A$, choose a point $v\in A$ lying in the opposite quadrant. Then it holds $f_{p_1}(u)+f_{p_1}(v)=d_1(u,v)$, which shows that $f_{p_1}$ belongs to $T(A)$.
\item[ii)]
Let $p_1=(a+t,b)$. In this case all points of $A$ except the point $(a,b)$, belong to the quadrant $Q_{p_1}^{++}$. Now if $u\in A$ is given from this quadrant, one can choose $v=(a,b)$ and if $u\in A$ is given as $u=(a,b)$, then one can choose $v\in A$ from $Q_{p_1}^{++}$ and it holds $f_{p_1}(u)+f_{p_1}(v)=d_1(u,v)$, so that $f_{p_1}$ belongs to $T(A)$.
\item[iii)]
Let $p_1=(a-t,b)$. In this case all points of $A$ except the point $(a,b)$, belong to the quadrant $Q_{p_1}^{-+}$. Now if $u\in A$ is given from this quadrant, one can choose $v=(a,b)$ and if $u\in A$ is given as $u=(a,b)$, then one can choose $v\in A$ from $Q_{p_1}^{-+}$ and it holds $f_{p_1}(u)+f_{p_1}(v)=d_1(u,v)$, so that $f_{p_1}$ belongs to $T(A)$.
\end{itemize}

Let us denote the auxiliary points emerging by application of Step 3 successively by $p_1, p_2, \cdots, p_n$. Assuming that the subset $A\cup\{p_1,p_2,\cdots,p_k\}\subseteq\mathbb{R}_1^2$ can be embedded into $T(A)$, we will show that $A\cup\{p_1,p_2,\cdots,p_{k+1}\}$ can be embedded into $T(A)$.

Note that our assumption implies \mbox{$A\cup\{p_1,p_2,\cdots,p_k\}\cong T(A)$.}

Let us denote $p_k=(c,d)$. Let us define the map
\[
f_{p_{k+1}}:A\cup\{p_1,p_2,\cdots,p_k\}\rightarrow\mathbb{R}^{\geq0}, \, f_{p_{k+1}}(q)=d_1(p_{k+1},q).
\]

We will show that this function belongs to $T(A\cup\{p_1,p_2,\cdots,p_k\})$. We have obviously $f_{p_{k+1}}(q)+f_{p_{k+1}}(r)\geq d_1(q,r)$ for any $q,r\in A\cup\{p_1,p_2,\cdots,p_k\}$ and we want to show that for any $u\in A\cup\{p_1,p_2,\cdots,p_k\}$, there exists a $v\in A\cup\{p_1,p_2,\cdots,p_k\}$ such that $f_{p_{k+1}}(u)+f_{p_{k+1}}(v)=d_1(u,v)$

We consider again the three positions of $p_{k+1}$ separately.

\begin{itemize}
\item[i)]
Let $p_{k+1}=(c,d+t)$. In this case every quadrant of the point $p_{k+1}$ contains a point of $A\cup\{p_1,p_2,\cdots,p_k\}$ and for any point $u\in A\cup\{p_1,p_2,\cdots,p_k\}$, one can choose a point $v$ in the opposite quadrant as in the previous case.
\item[ii)]
Let $p_{k+1}=(c+t,d)$. In this case there can not be any points of $A\cup\{p_1,p_2,\cdots,p_k\}$ in the interior of $Q_{p_{k+1}}^{-+}$. Now if $u$ is given from $Q_{p_{k+1}}^{++}$ or $Q_{p_{k+1}}^{+-}$, we can choose $v=(c,d)$; if $u$ is given from $Q_{p_{k+1}}^{--}$, we can choose $v\in Q_{p_{k+1}}^{++}\cap(A\cup\{p_1,p_2,\cdots,p_k\})$ and in both cases $f_{p_{k+1}}(u)+f_{p_{k+1}}(v)=d_1(u,v)$ is satisfied.
\item[iii)]
Let $p_{k+1}=(c-t,d)$. In this case there can not be any points of $A\cup\{p_1,p_2,\cdots,p_k\}$ in the interior of $Q_{p_{k+1}}^{++}$. Now if $u$ is given from $Q_{p_{k+1}}^{-+}$ or $Q_{p_{k+1}}^{--}$, we can choose $v=(c,d)$; if $u$ is given from $Q_{p_{k+1}}^{+-}$, we can choose $v\in Q_{p_{k+1}}^{-+}\cap(A\cup\{p_1,p_2,\cdots,p_k\})$ and in both cases $f_{p_{k+1}}(u)+f_{p_{k+1}}(v)=d_1(u,v)$ is satisfied.
\end{itemize}

The above consideration show that $T(A\cup\{p_1,p_2,\cdots,p_n\})=T(A)$ and to determine $T(A)$, it will be enough to determine $T(A\cup\{p_1,p_2,\cdots,p_n\})$. Since the tight span is strictly intrinsic (see \cite{kilicarxivilk}), an embedded copy of $T(A\cup\{p_1,p_2,\cdots,p_n\})$ must contain the ${\rm SPINE}$ and by the same reason the ${\rm SKELETON}$, so that we get
\begin{eqnarray*}
T(A)&=&T(A\cup\{p_1,p_2,\cdots,p_n\})\\
&=&T(A\cup {\rm SPINE})\\
&=&T({\rm SKELETON}).
\end{eqnarray*}
As the ${\rm SKELETON}$ is a path connected subset and $L_x-$invariant, we get by \cite{kilicarxiv}
\[
T(A)=\overline{L_y({\rm SKELETON})}=L_y({\rm SKELETON}),
\]
which is the conclusion of the theorem.
\end{proof}

\section{Some Examples}\label{Section2}

\begin{example}
Let $A=\{(-1,0), (0,-2),(\frac{3}{2},1)\}$. The ${\rm SPINE}$ and the ${\rm SKELETON}$ for this set are shown in Figure~\ref{Fig:ed3points}b and Figure~\ref{Fig:ed3points}c. The  ${\rm SKELETON}$ is at the same time the tight span as the vertical hatching adds no points. We remind that the tight span is unique only up to isometry and there can be different embeddings of a tight span into the Manhattan plane. Two additional such embeddings are shown in Figures~\ref{Fig:ed3points}d--\ref{Fig:ed3points}e.

We also want to remark that the assumption of Lemma~9 of Eppstein \cite{epp} is not satisfied (i.e. the orthogonal hull is not connected), so that the orthogonal hull of the set $A$ does not give the tight span of $A$ (see Figure~\ref{Fig:ed3points}f).
\end{example}

In the following examples, Examples~2-5, we show in each case the considered subset of the Manhattan plane, its ${\rm SPINE}$ and ${\rm SKELETON}$ according to our algorithm and the vertical hatching of the ${\rm SKELETON}$, which is the tight span of the given finite subset (see Figures~\ref{Fig:4points}~-~\ref{Fig:ed15points}).

\begin{figure}[h!]
\begin{minipage}[b]{0.45\linewidth}
\centering
\includegraphics[width=0.7\linewidth]{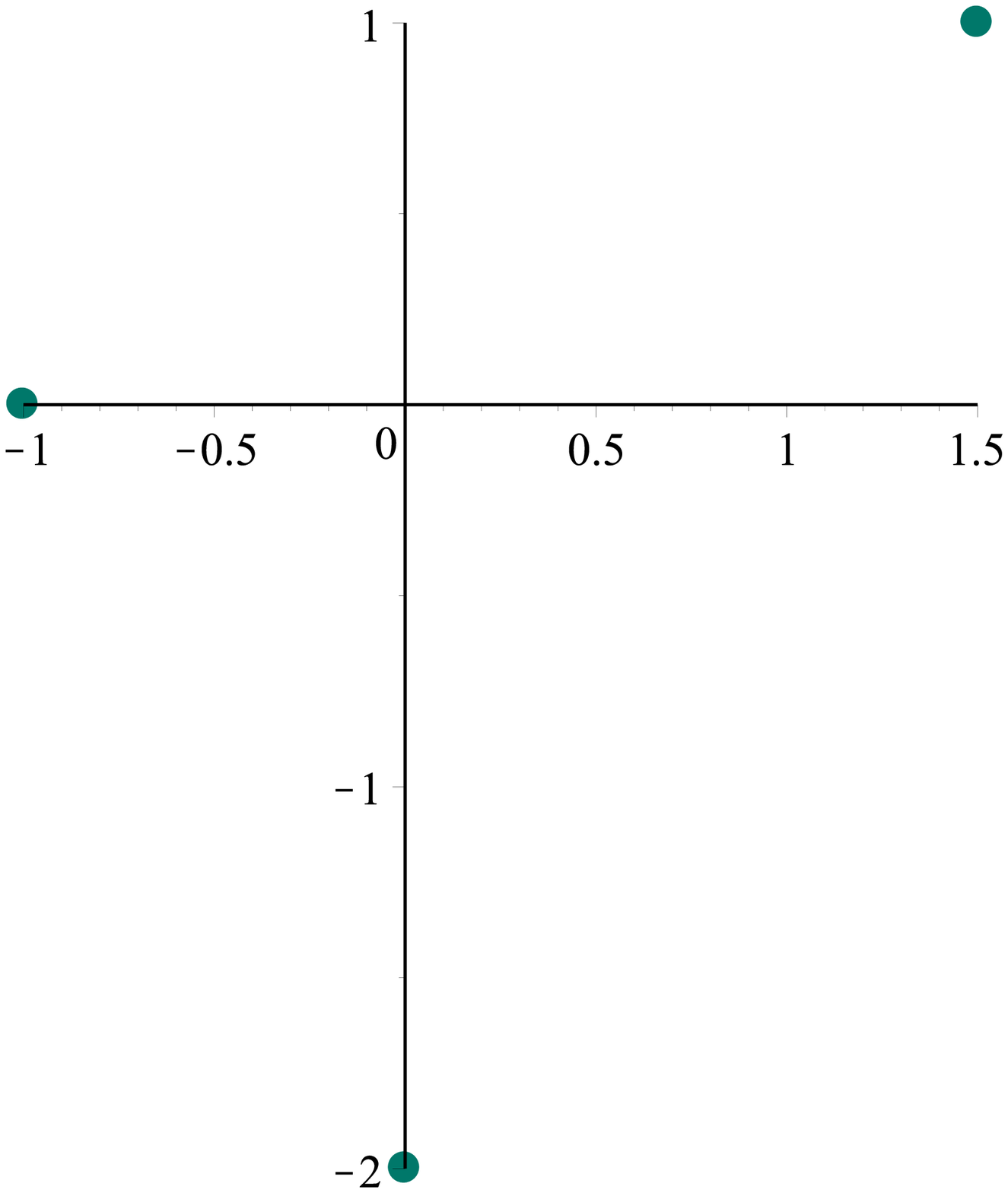}
\end{minipage}
\hspace{1cm}
\begin{minipage}[b]{0.45\linewidth}
\centering
\includegraphics[width=0.7\linewidth]{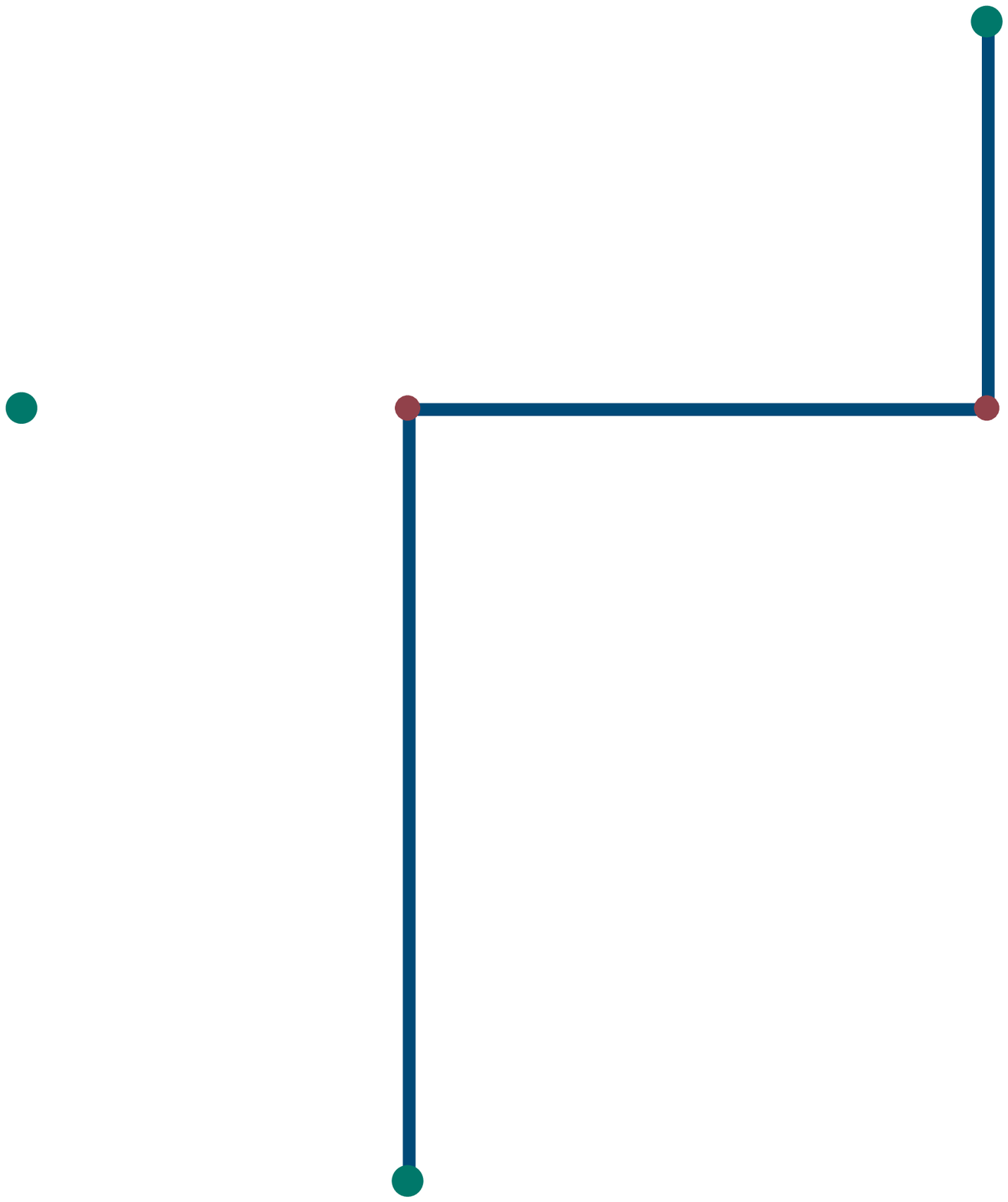}
\end{minipage}
\newline
\begin{minipage}[b]{0.45\linewidth}
\centering
a) The $3$-point set $A$
\newline
\end{minipage}
\hspace{1cm}
\begin{minipage}[b]{0.45\linewidth}
\centering
b) The ${\rm SPINE}$ of $A$
\newline
\end{minipage}
\newline

\begin{minipage}[b]{0.45\linewidth}
\centering
\includegraphics[width=0.7\linewidth]{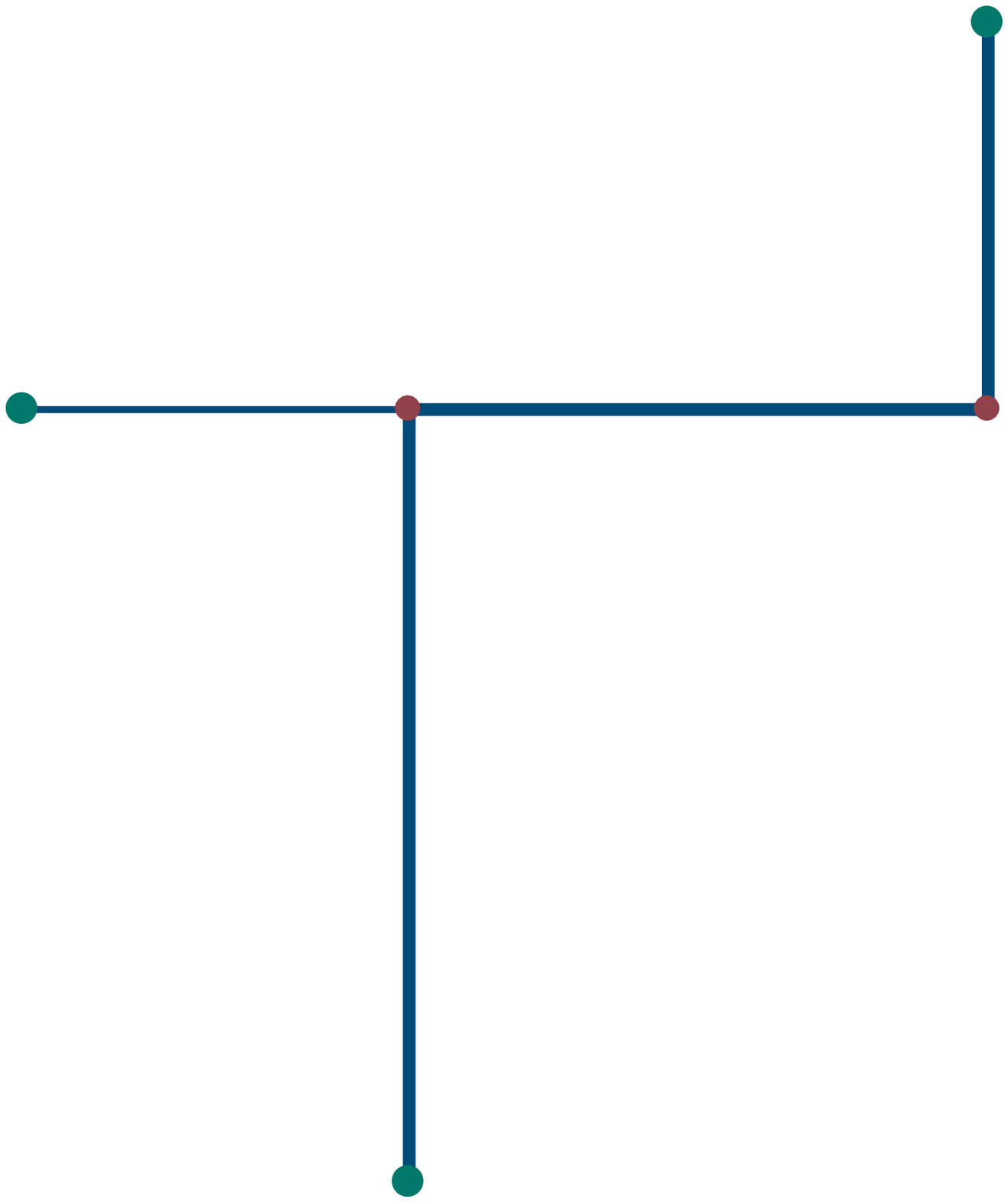}
\end{minipage}
\hspace{1cm}
\begin{minipage}[b]{0.45\linewidth}
\centering
\includegraphics[width=0.575\linewidth]{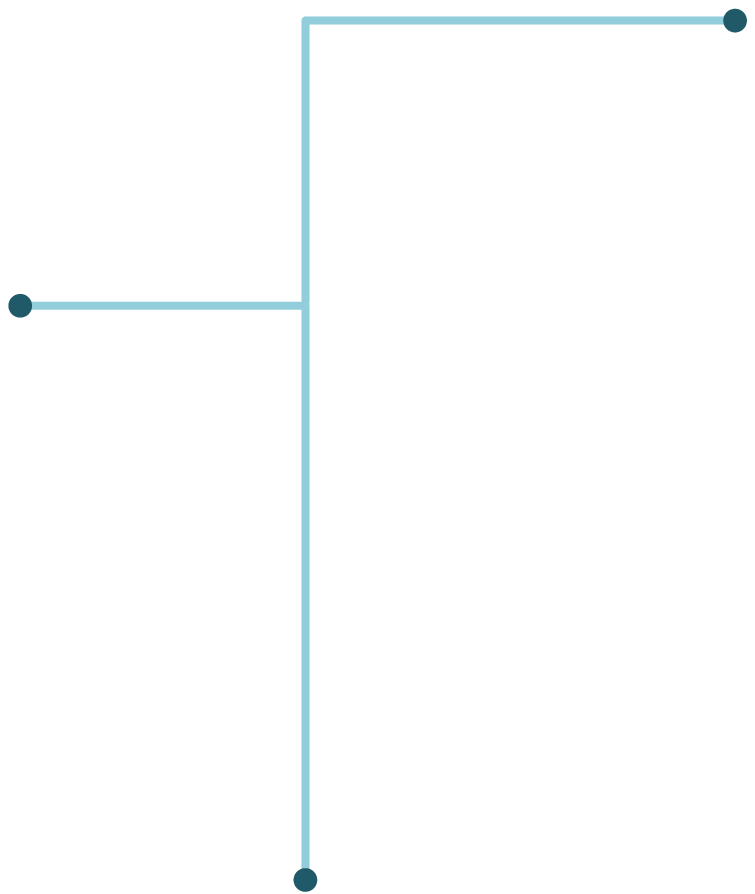}
\end{minipage}
\newline

\begin{minipage}[b]{0.45\linewidth}
\centering
c) The ${\rm SKELETON}$ and the tight span of $A$
\newline
\end{minipage}
\hspace{1cm}
\begin{minipage}[b]{0.45\linewidth}
\centering
d) Another embedding of the tight span of $A$
\newline
\end{minipage}
\newline

\begin{minipage}[b]{0.45\linewidth}
\centering
\includegraphics[width=0.575\linewidth]{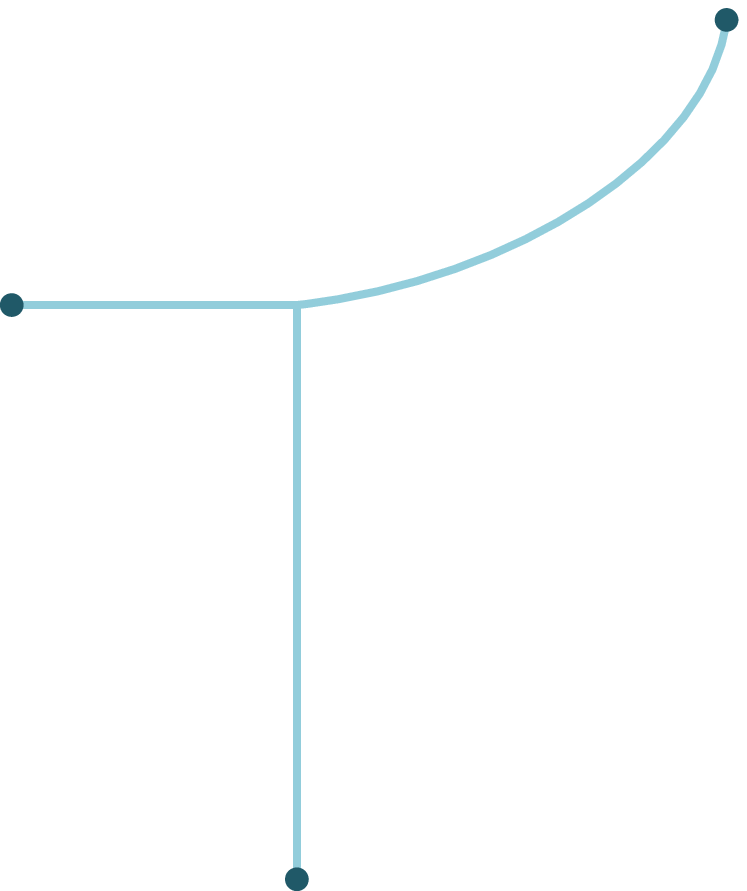}
\end{minipage}
\hspace{1cm}
\begin{minipage}[b]{0.45\linewidth}
\centering
\includegraphics[width=0.575\linewidth]{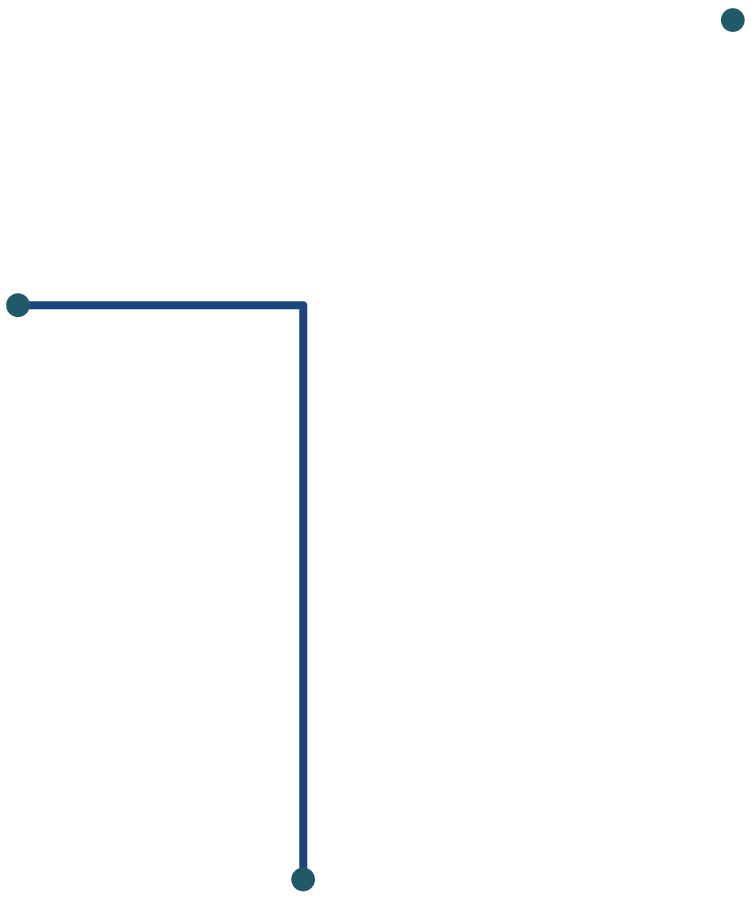}
\end{minipage}
\newline

\begin{minipage}[b]{0.45\linewidth}
\centering
e) Still another embedding of the tight span of $A$
\newline
\newline
\end{minipage}
\hspace{1cm}
\begin{minipage}[b]{0.45\linewidth}
\centering
f) The orthogonal hull of $A$ in the sense of Eppstein, which is not connected and therefore does not give the tight span of $A$
\end{minipage}
\newline

\caption{Tight span of a $3$-point set}\label{Fig:ed3points}
\end{figure}

\clearpage

\begin{example}
\mbox{  }
\begin{figure}[h!]
\begin{minipage}[b]{0.45\linewidth}
\centering
\includegraphics[width=0.75\linewidth]{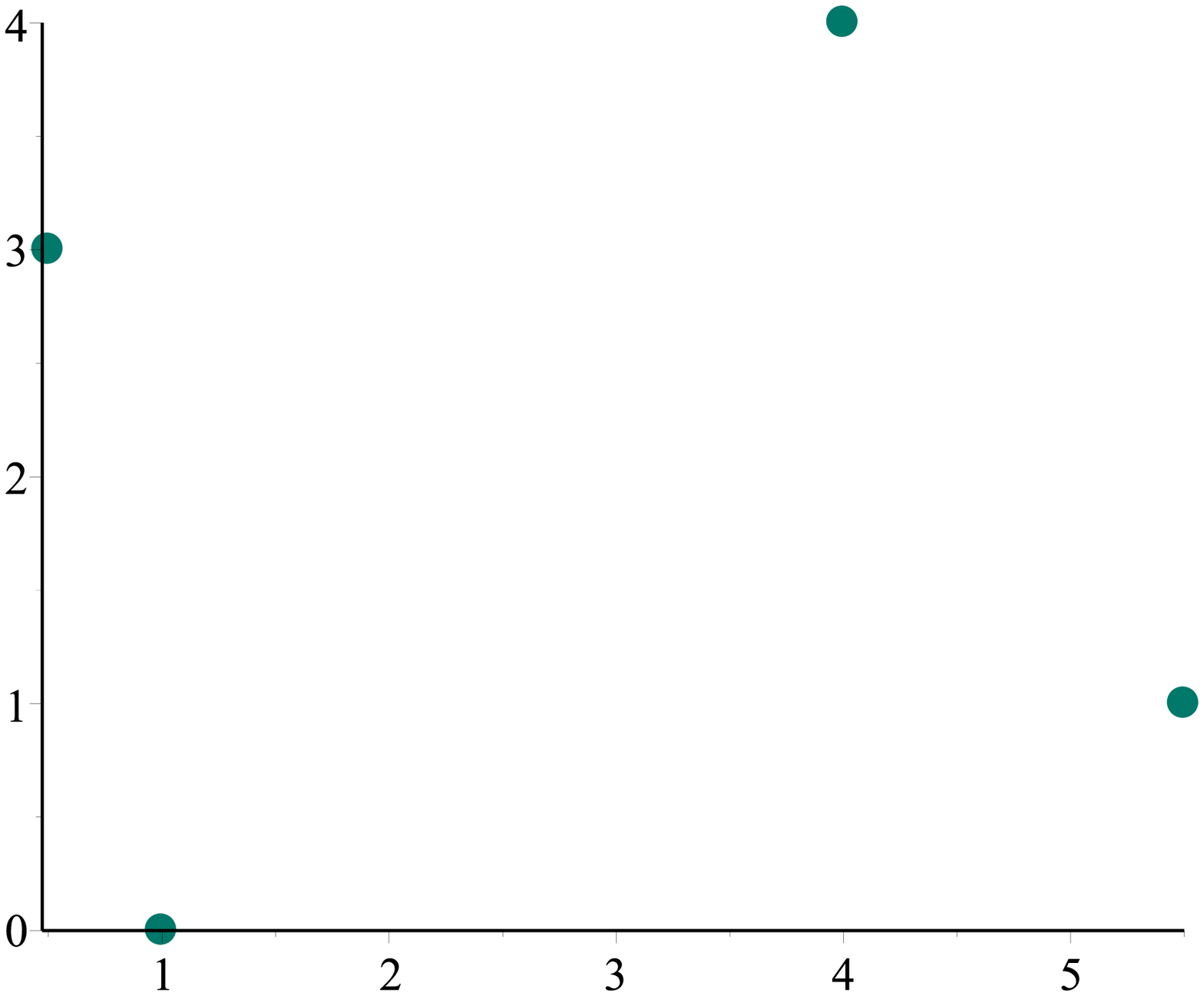}
\end{minipage}
\hspace{1cm}
\begin{minipage}[b]{0.45\linewidth}
\centering
\includegraphics[width=0.75\linewidth]{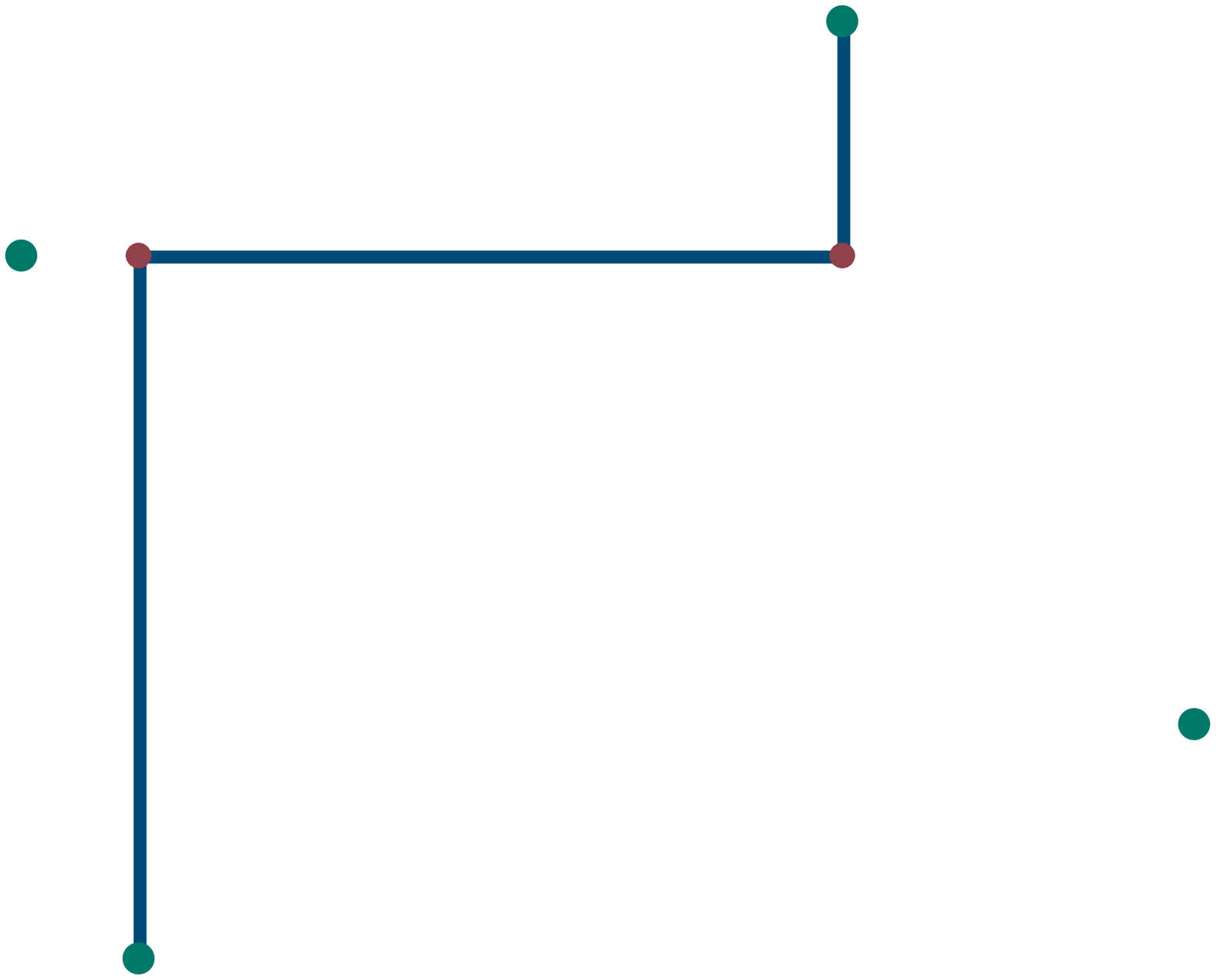}
\end{minipage}
\newline
\begin{minipage}[b]{0.45\linewidth}
\centering
(a) The $4$-point set $A$
\end{minipage}
\hspace{1cm}
\begin{minipage}[b]{0.45\linewidth}
\centering
(b) The {\rm SPINE} of $A$
\end{minipage}
\newline\newline
\begin{minipage}[b]{0.45\linewidth}
\centering
\includegraphics[width=0.75\linewidth]{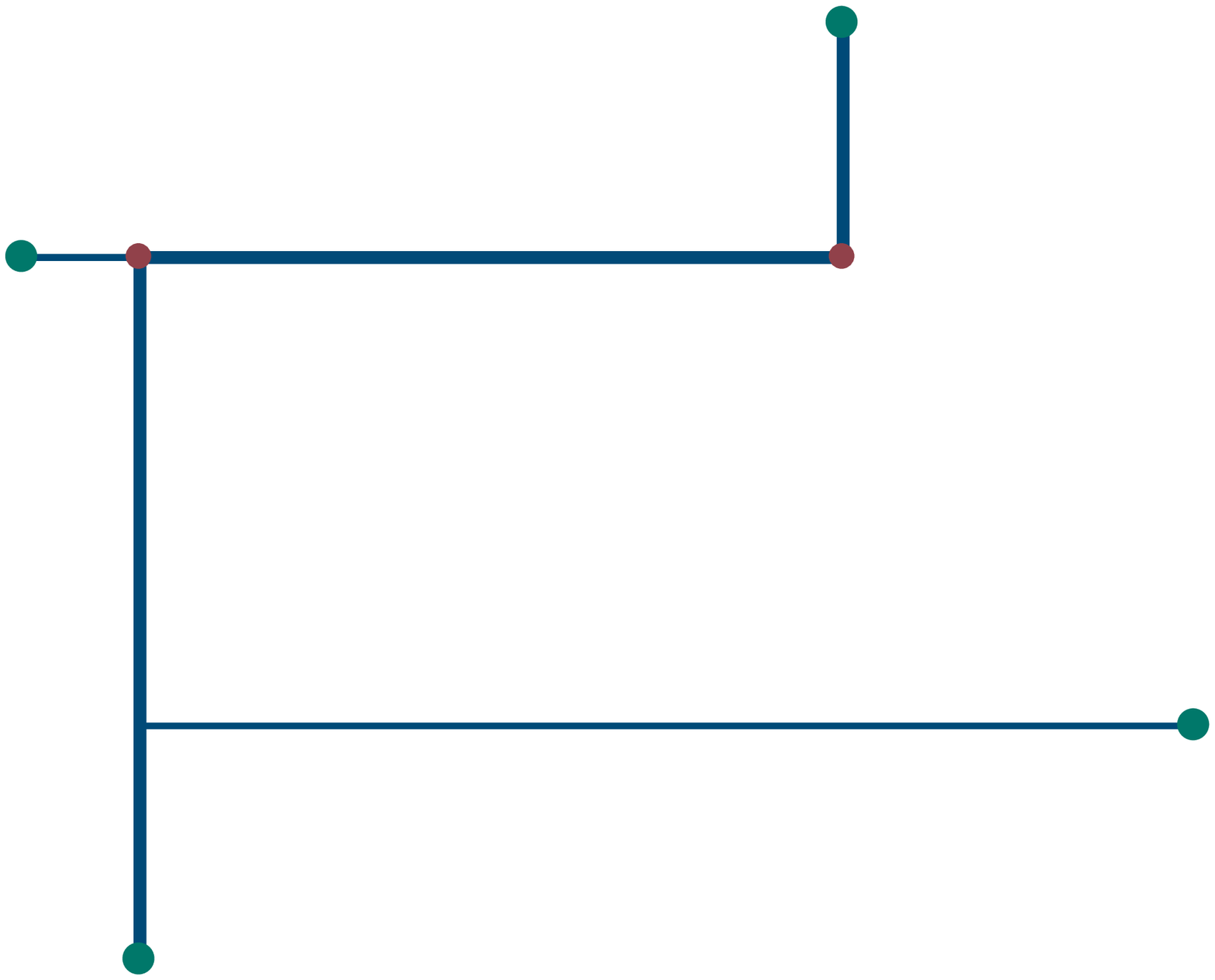}
\end{minipage}
\hspace{1cm}
\begin{minipage}[b]{0.45\linewidth}
\centering
\includegraphics[width=0.75\linewidth]{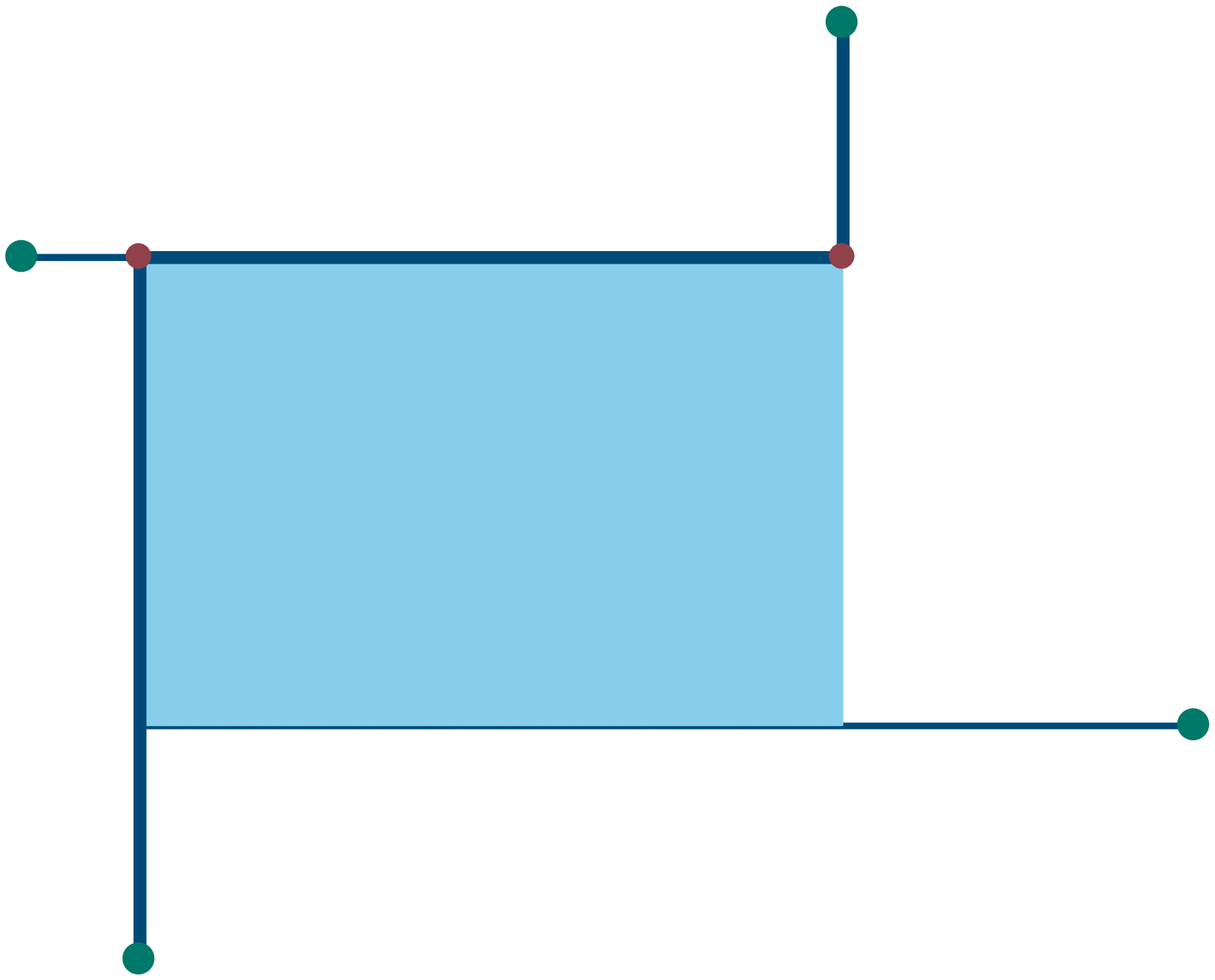}
\end{minipage}
\newline
\begin{minipage}[b]{0.45\linewidth}
\centering
(c) The {\rm SKELETON} of $A$
\newline
\end{minipage}
\hspace{1cm}
\begin{minipage}[b]{0.45\linewidth}
\centering
(d) The vertical hatching of the {\rm SKELETON}
\end{minipage}
\newline\newline
\begin{minipage}[b]{0.90\linewidth}
\centering
\includegraphics[width=0.375\linewidth]{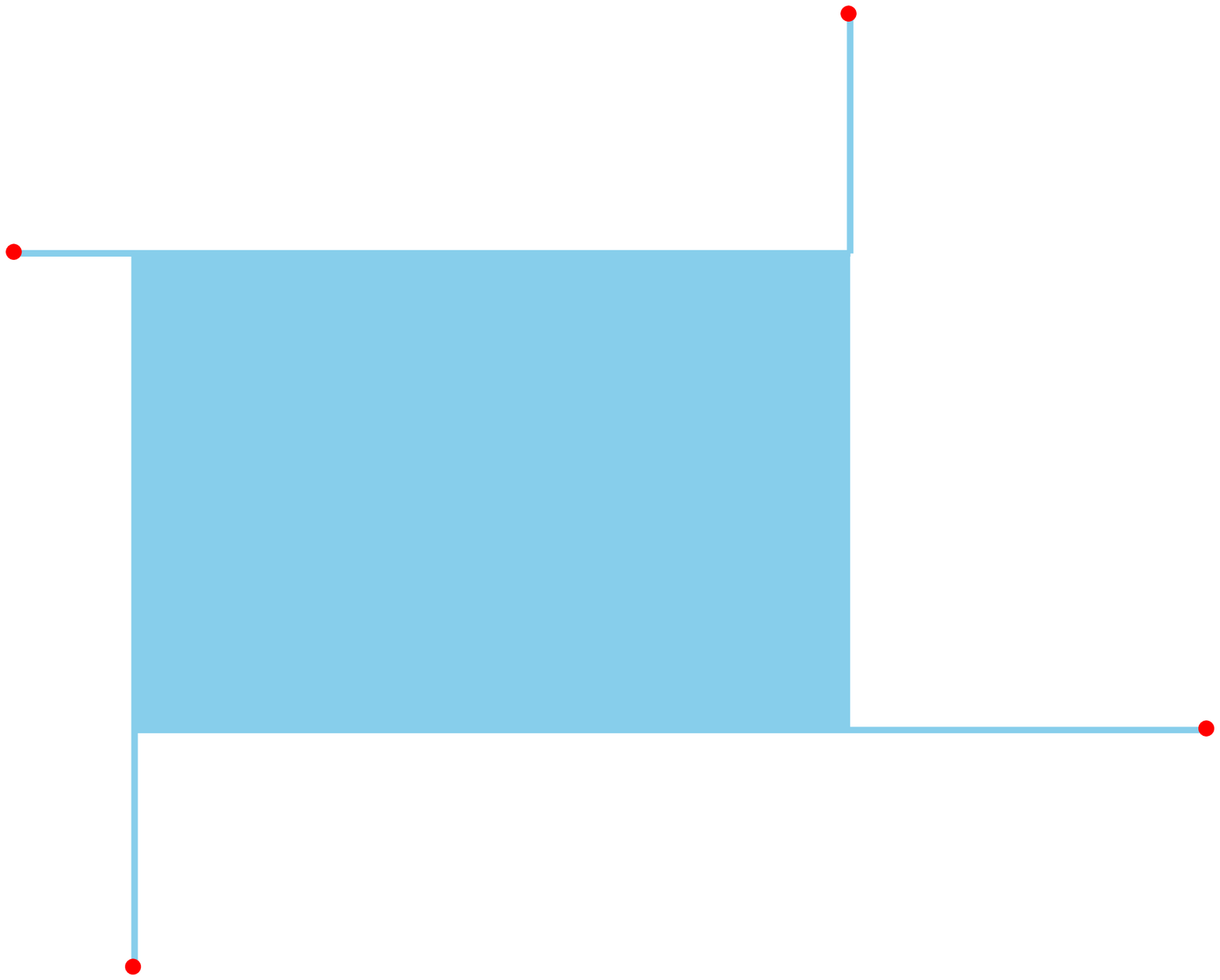}
\end{minipage}
\newline
\begin{minipage}[b]{0.95\linewidth}
\centering
(e) The tight span of the finite set $A$
\end{minipage}
\vspace{0.25cm}
\caption{Construction of the tight span of a $4$-point set given in Figure~\ref{Fig:4points}a}\label{Fig:4points}
\end{figure}
\end{example}

\clearpage

\begin{example}
\mbox{ }
\begin{figure}[h!]
\begin{minipage}[b]{0.45\linewidth}
\centering
\includegraphics[width=0.75\linewidth]{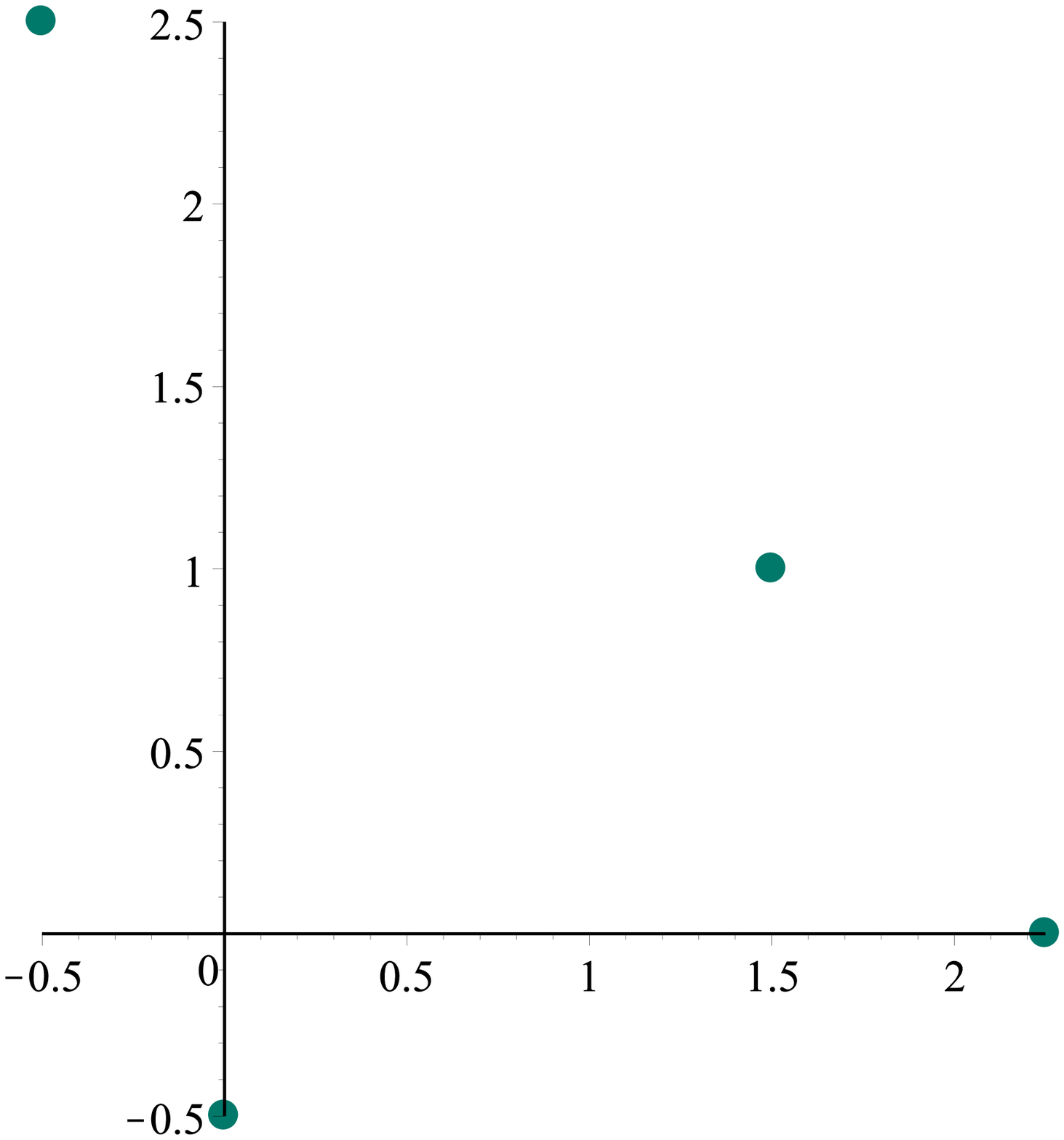}
\end{minipage}
\hspace{1cm}
\begin{minipage}[b]{0.45\linewidth}
\centering
\includegraphics[width=0.75\linewidth]{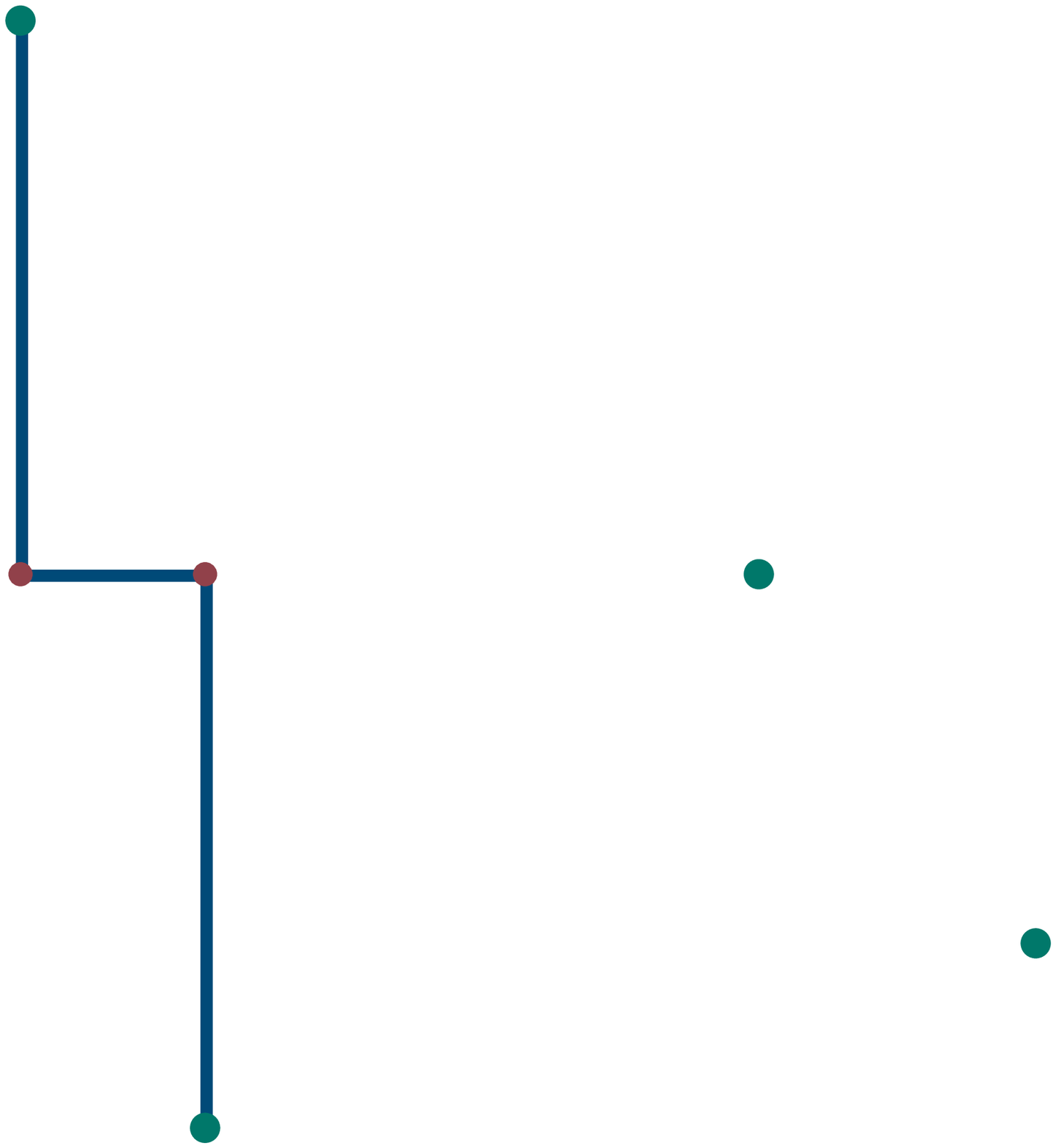}
\end{minipage}
\newline
\begin{minipage}[b]{0.45\linewidth}
\centering
(a) The $4$-point set $A$
\end{minipage}
\hspace{1cm}
\begin{minipage}[b]{0.45\linewidth}
\centering
(b) The {\rm SPINE} of $A$
\end{minipage}
\newline\newline
\begin{minipage}[b]{0.45\linewidth}
\centering
\includegraphics[width=0.75\linewidth]{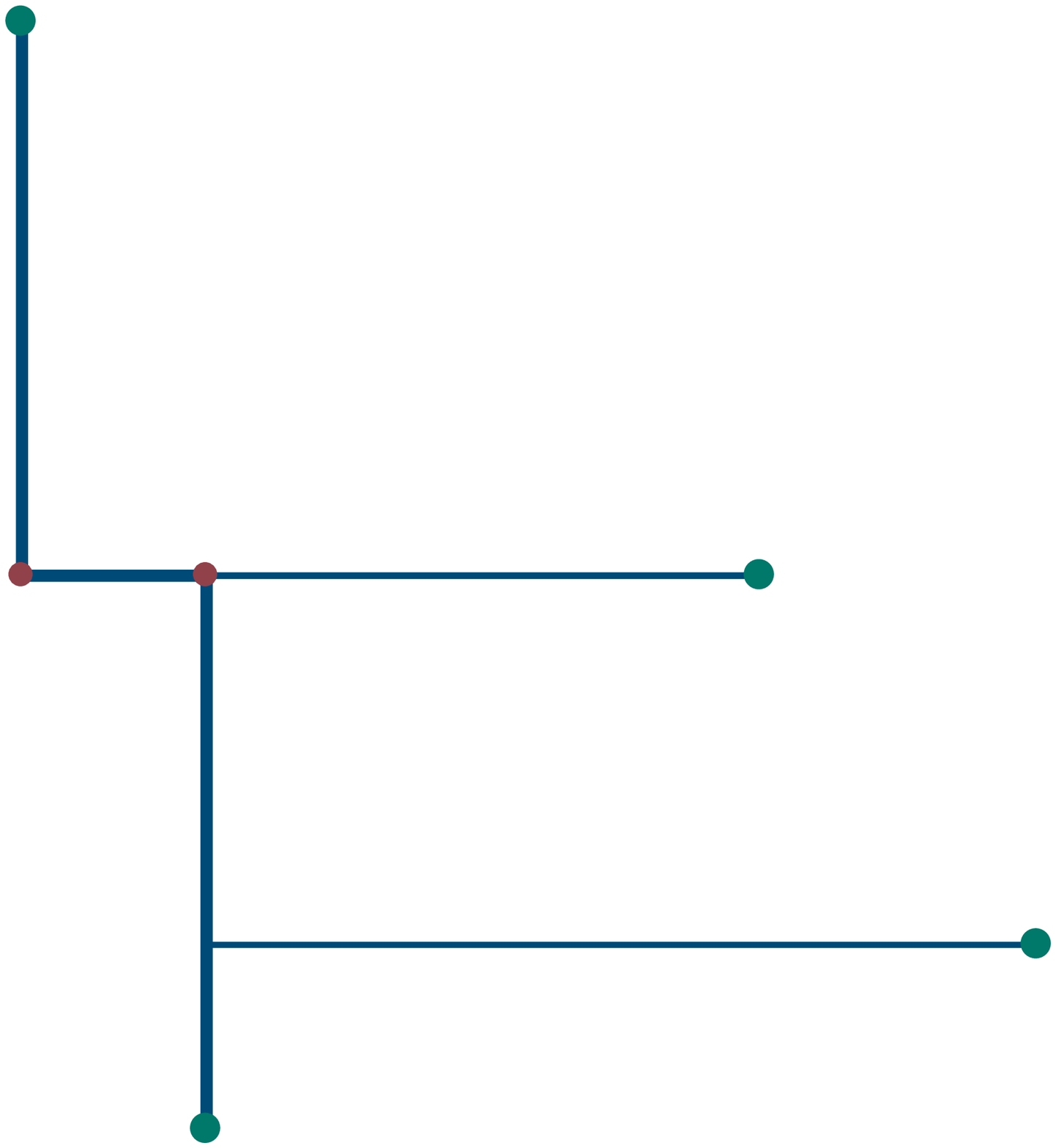}
\end{minipage}
\hspace{1cm}
\begin{minipage}[b]{0.45\linewidth}
\centering
\includegraphics[width=0.75\linewidth]{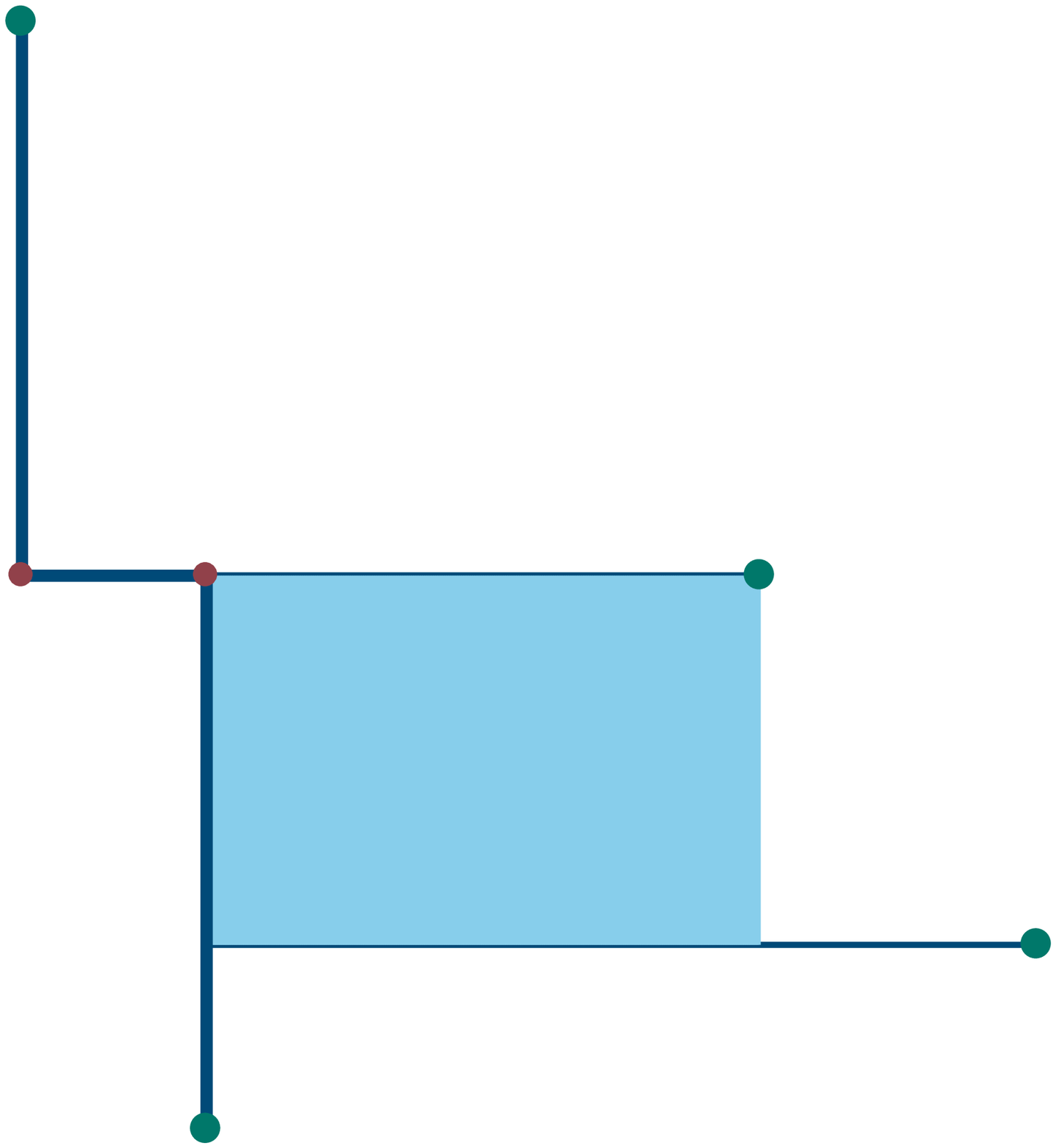}
\end{minipage}
\newline
\begin{minipage}[b]{0.45\linewidth}
\centering
(c) The {\rm SKELETON} of $A$
\newline
\end{minipage}
\hspace{1cm}
\begin{minipage}[b]{0.45\linewidth}
\centering
(d) The vertical hatching of the {\rm SKELETON}
\end{minipage}
\newline\newline
\begin{minipage}[b]{0.90\linewidth}
\centering
\includegraphics[width=0.375\linewidth]{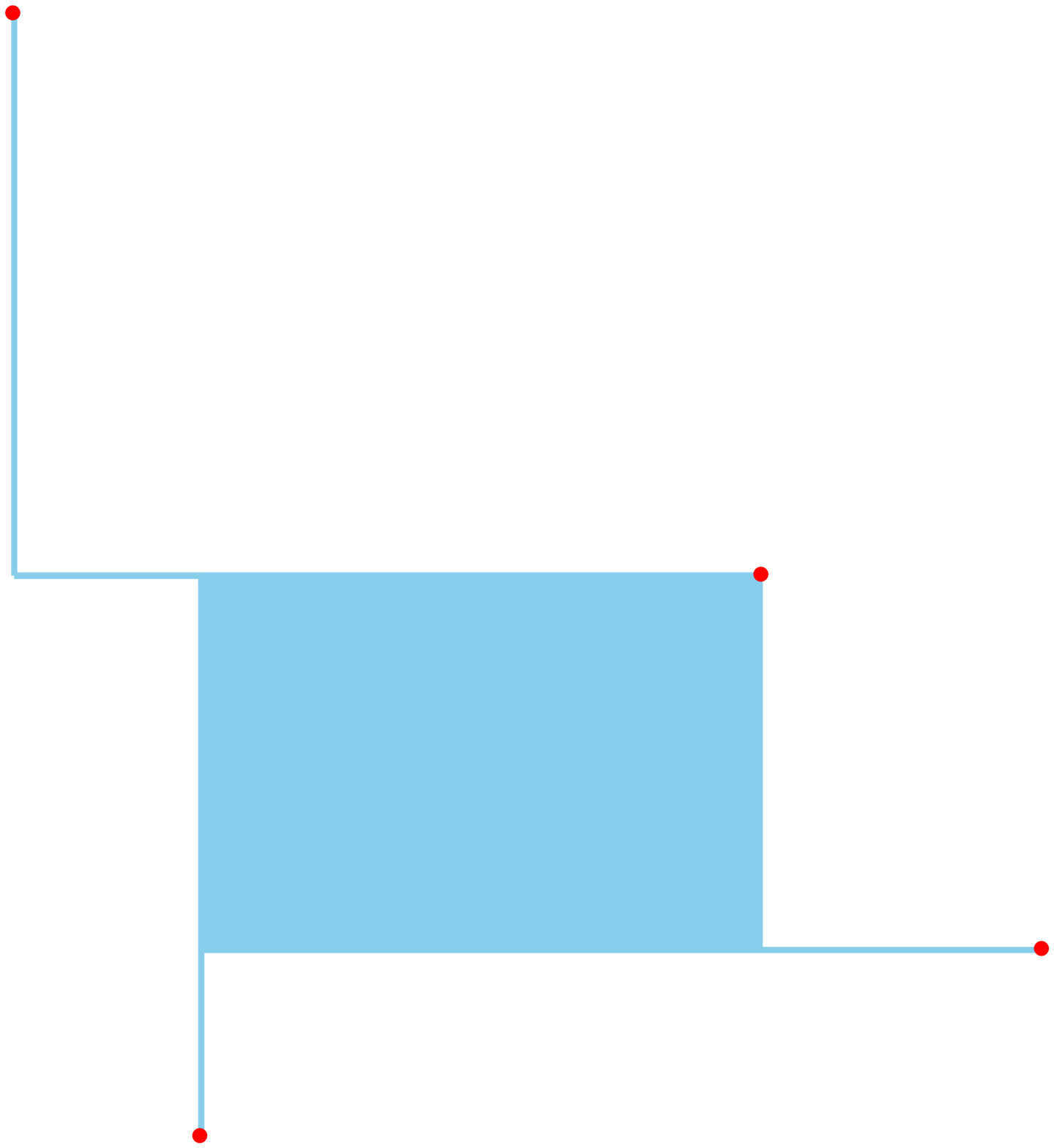}
\end{minipage}
\newline
\begin{minipage}[b]{0.95\linewidth}
\centering
(e) The tight span of the finite set $A$
\end{minipage}
\vspace{0.25cm}
\caption{Construction of the tight span of another $4$-point set given in Figure~\ref{Fig:ed4points}a}\label{Fig:ed4points}
\end{figure}
\end{example}

\clearpage

\begin{example}
\mbox{ }
\begin{figure}[h!]
\begin{minipage}[b]{0.45\linewidth}
\centering
\includegraphics[width=0.75\linewidth]{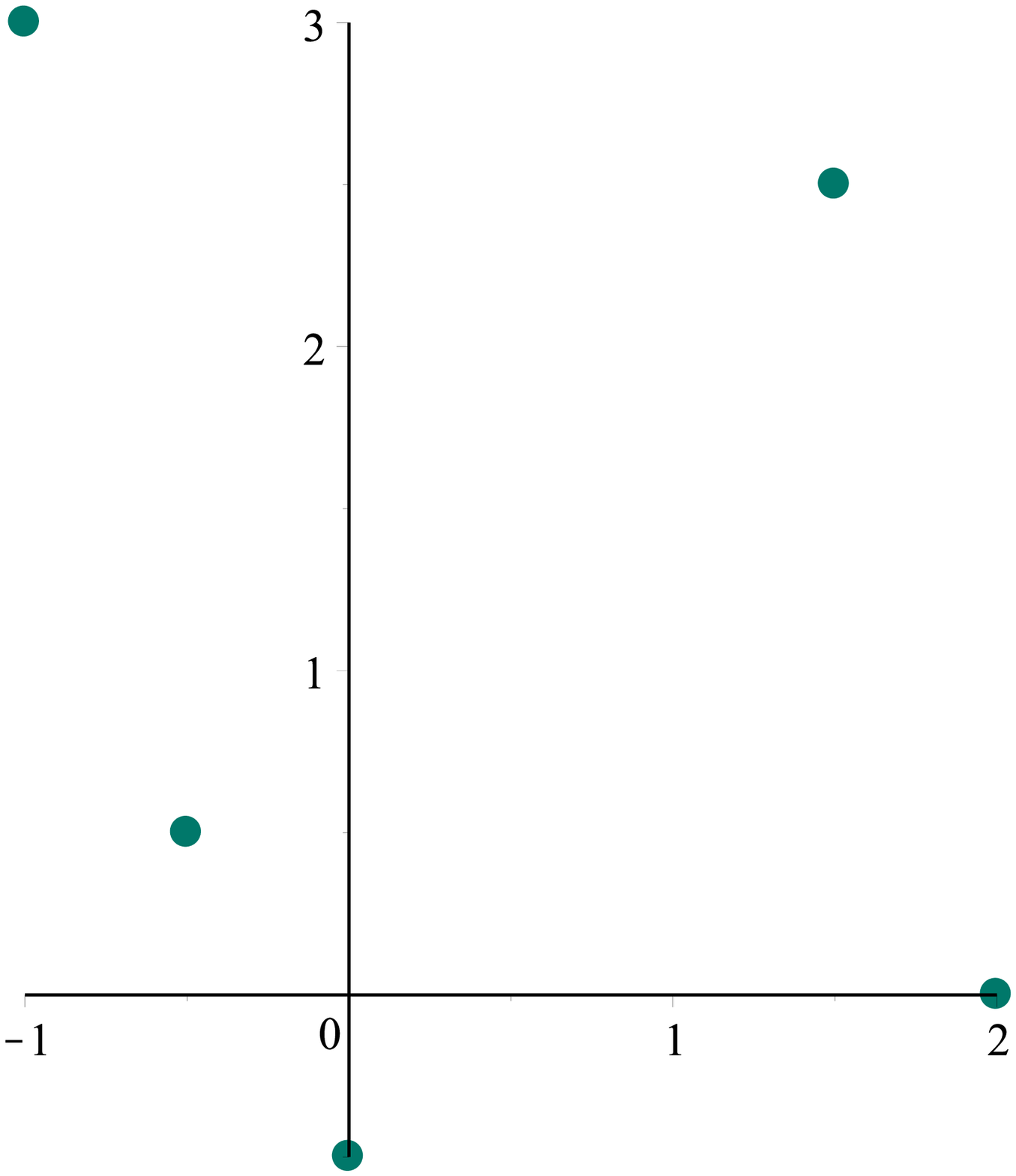}
\end{minipage}
\hspace{1cm}
\begin{minipage}[b]{0.45\linewidth}
\centering
\includegraphics[width=0.75\linewidth]{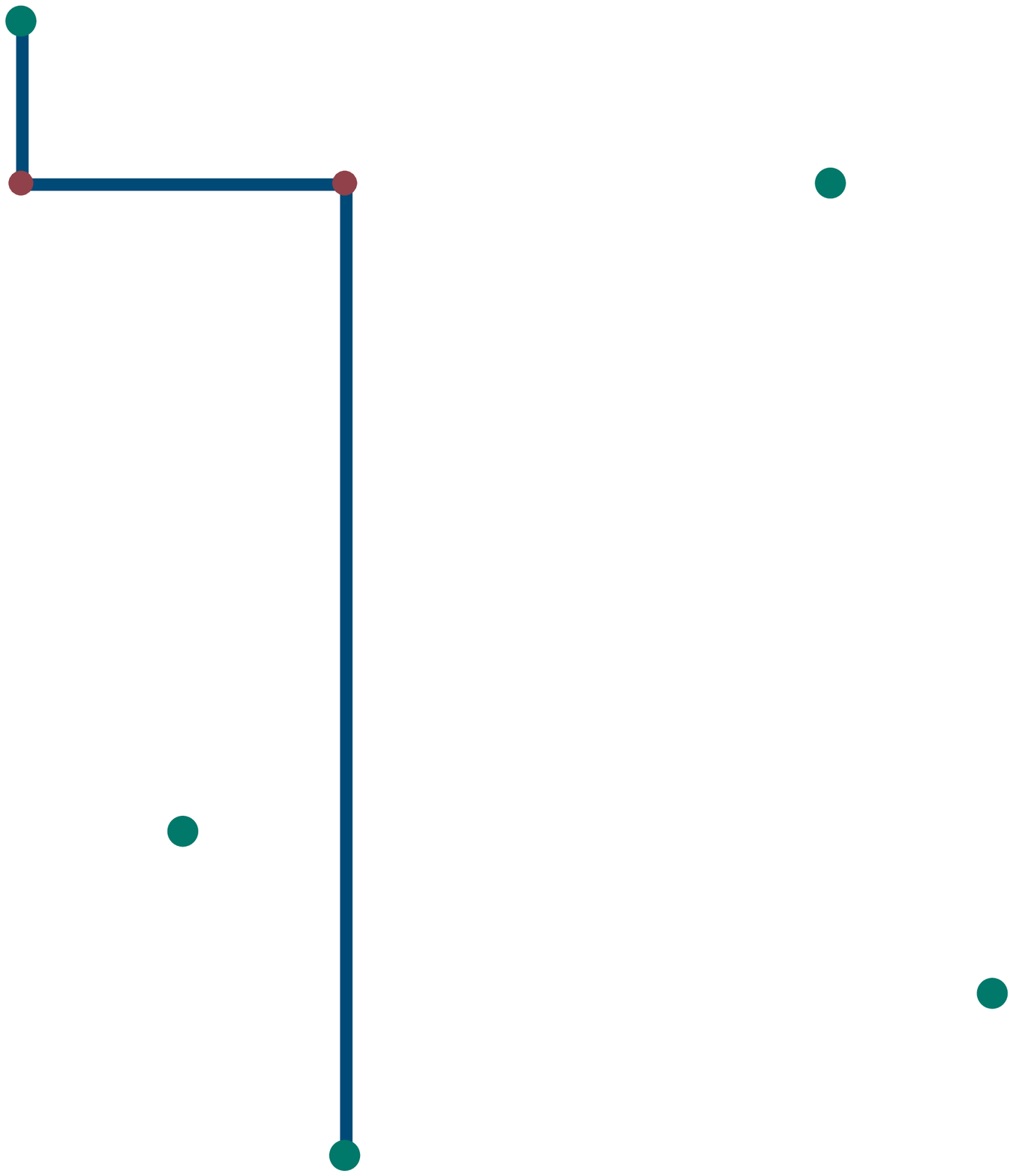}
\end{minipage}
\newline
\begin{minipage}[b]{0.45\linewidth}
\centering
(a) The $5$-point set $A$
\end{minipage}
\hspace{1cm}
\begin{minipage}[b]{0.45\linewidth}
\centering
(b) The {\rm SPINE} of $A$
\end{minipage}
\newline\newline
\begin{minipage}[b]{0.45\linewidth}
\centering
\includegraphics[width=0.75\linewidth]{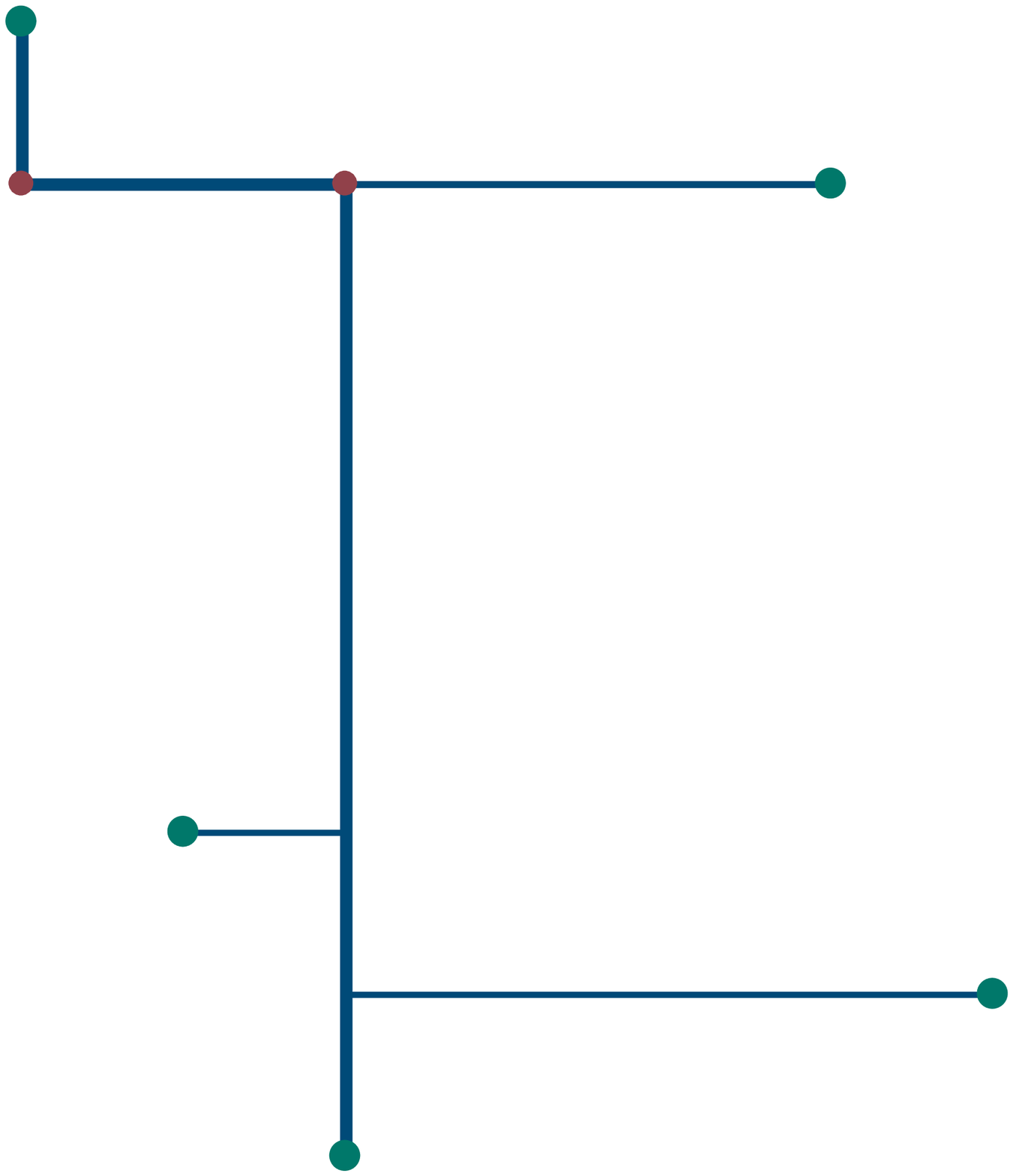}
\end{minipage}
\hspace{1cm}
\begin{minipage}[b]{0.45\linewidth}
\centering
\includegraphics[width=0.75\linewidth]{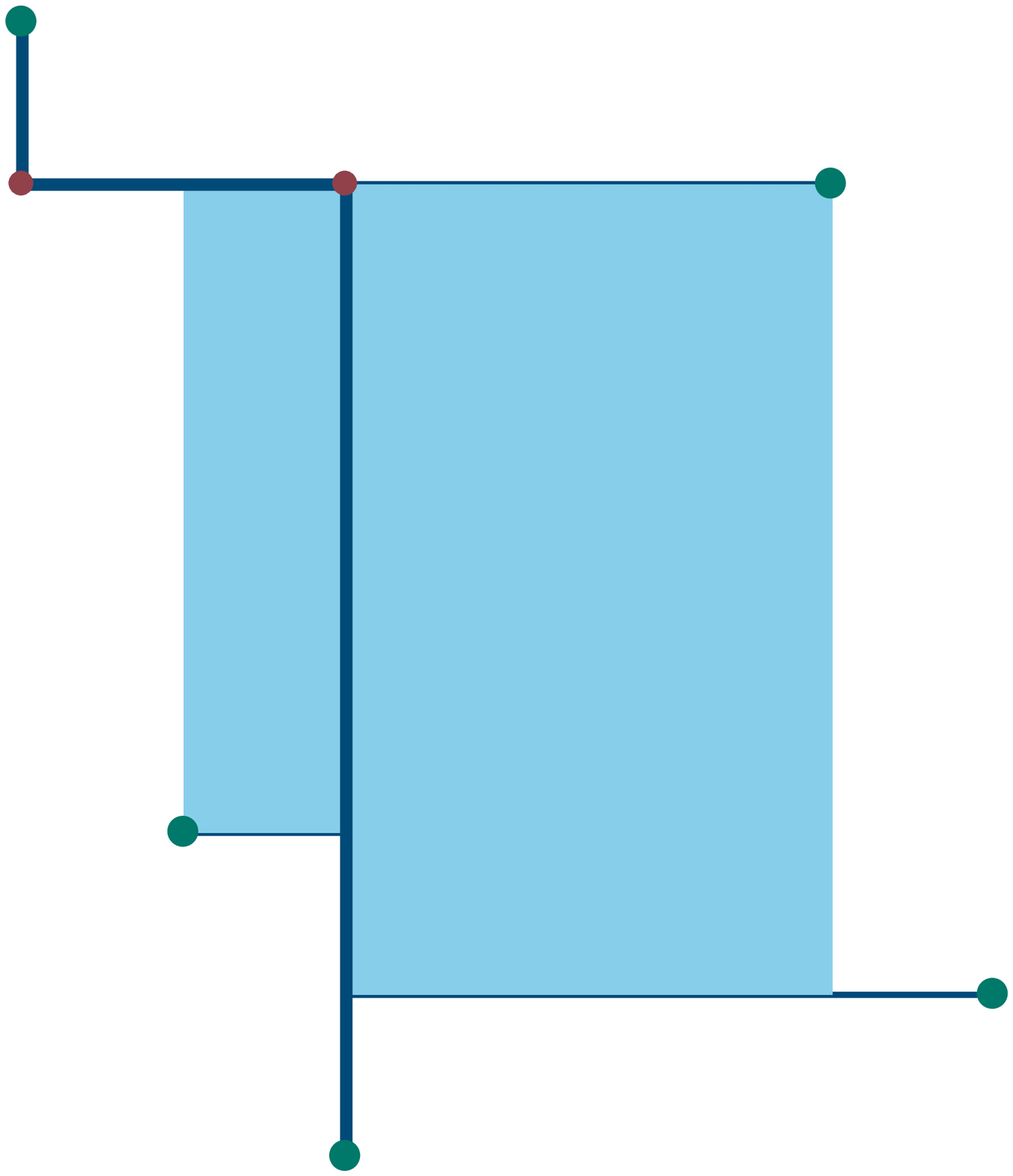}
\end{minipage}
\newline
\begin{minipage}[b]{0.45\linewidth}
\centering
(c) The {\rm SKELETON} of $A$
\newline
\end{minipage}
\hspace{1cm}
\begin{minipage}[b]{0.45\linewidth}
\centering
(d) The vertical hatching of the {\rm SKELETON}
\end{minipage}
\newline\newline
\begin{minipage}[b]{0.90\linewidth}
\centering
\includegraphics[width=0.375\linewidth]{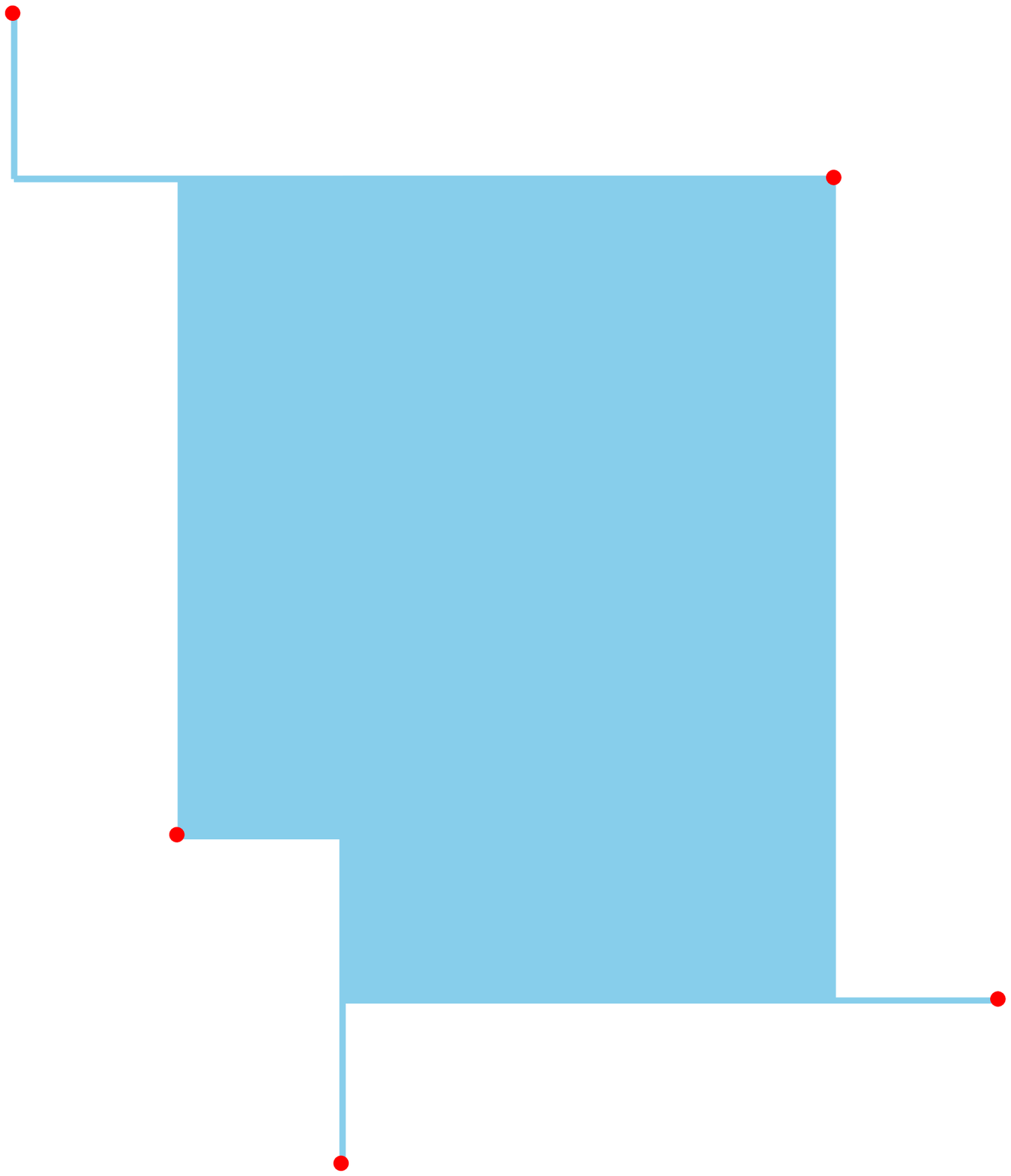}
\end{minipage}
\newline
\begin{minipage}[b]{0.95\linewidth}
\centering
(e) The tight span of the finite set $A$
\end{minipage}
\vspace{0.25cm}
\caption{Construction of the tight span of a $5$-point set given in Figure~\ref{Fig:ed5points}a}\label{Fig:ed5points}
\end{figure}
\end{example}

\clearpage
\begin{example}
\mbox{ }
\begin{figure}[h!]
\begin{minipage}[b]{0.45\linewidth}
\centering
\includegraphics[width=0.75\linewidth]{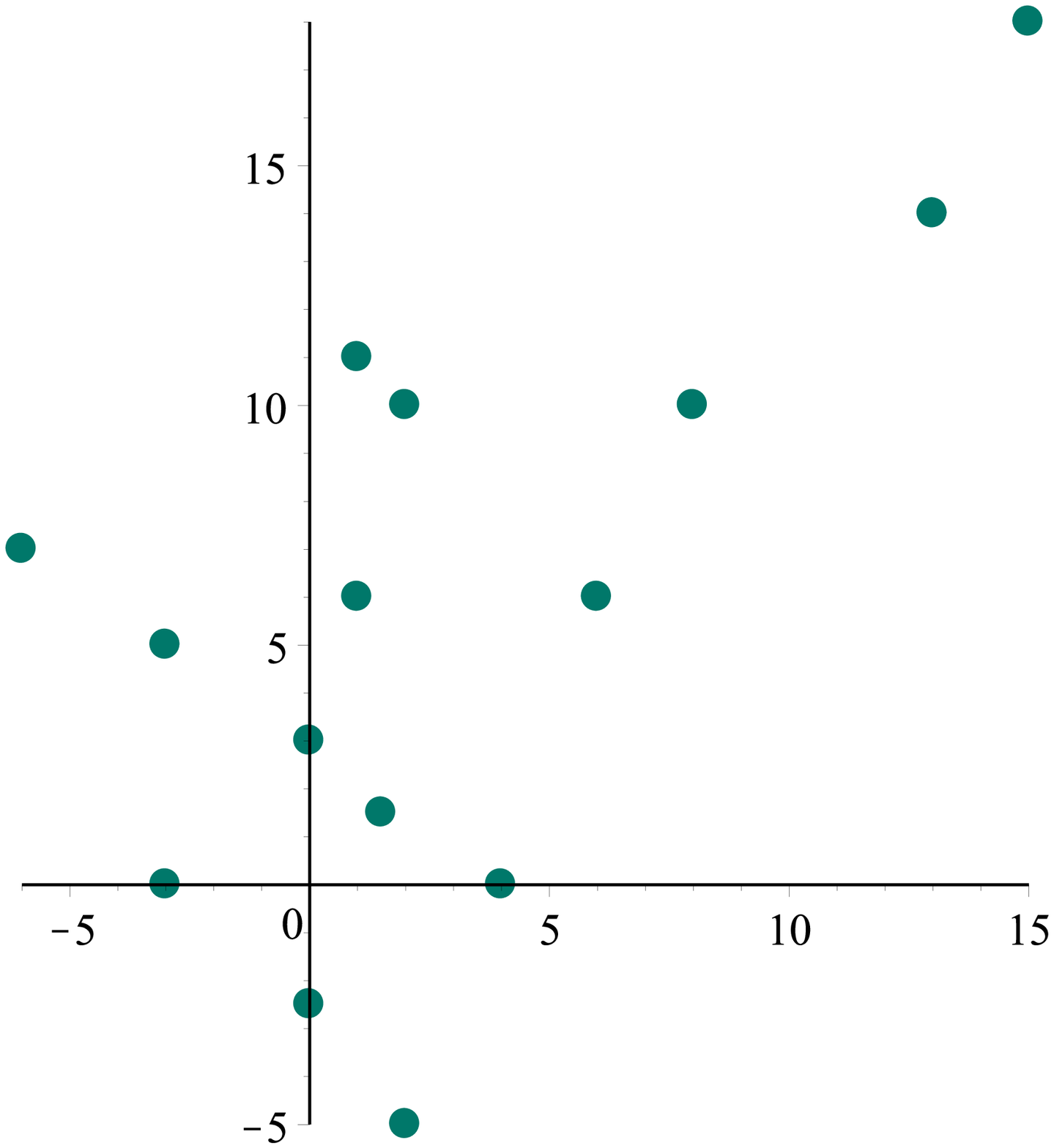}
\end{minipage}
\hspace{1cm}
\begin{minipage}[b]{0.45\linewidth}
\centering
\includegraphics[width=0.75\linewidth]{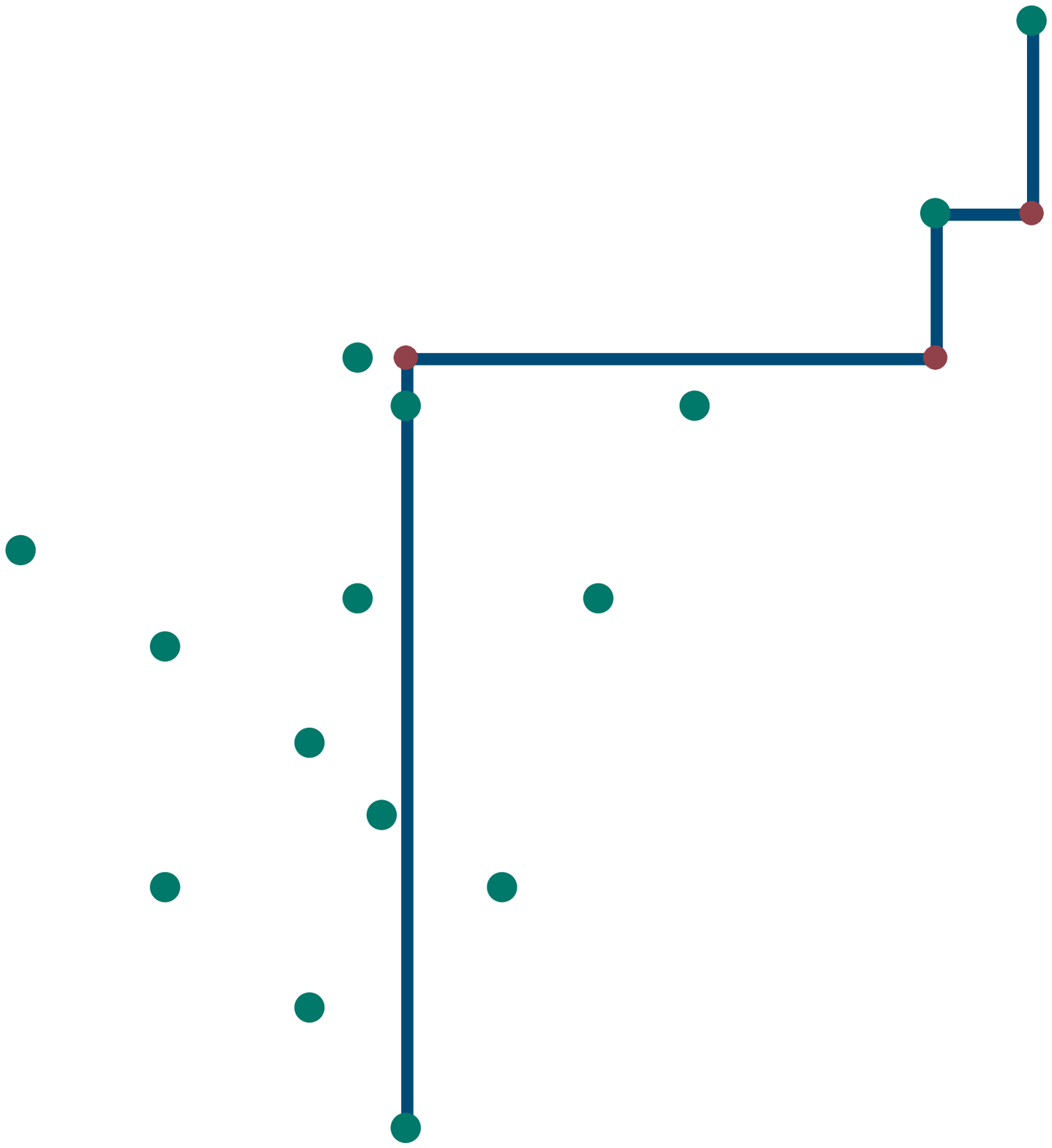}
\end{minipage}
\newline
\begin{minipage}[b]{0.45\linewidth}
\centering
(a) The $15$-point set $A$
\end{minipage}
\hspace{1cm}
\begin{minipage}[b]{0.45\linewidth}
\centering
(b) The {\rm SPINE} of $A$
\end{minipage}
\newline\newline
\begin{minipage}[b]{0.45\linewidth}
\centering
\includegraphics[width=0.75\linewidth]{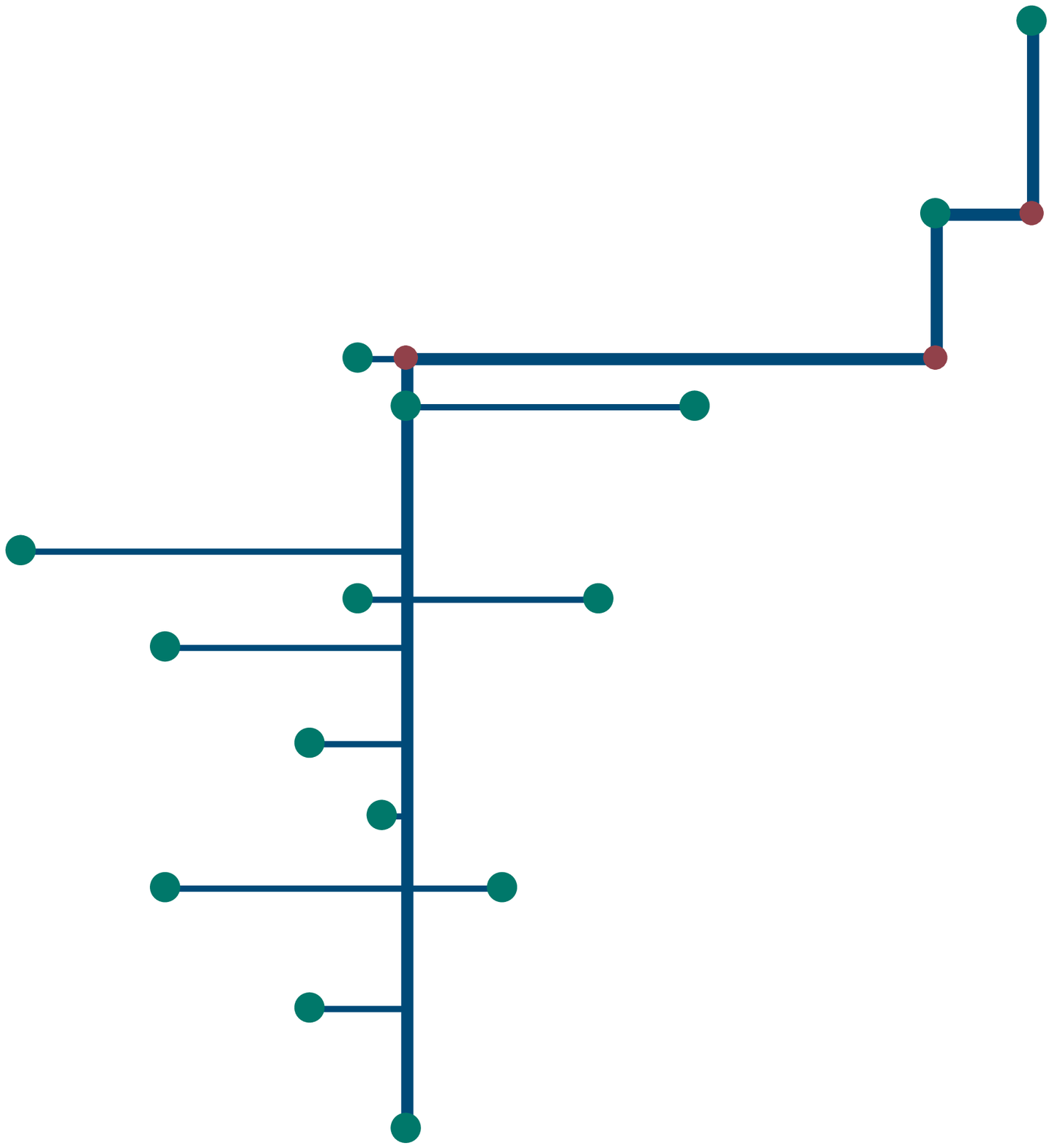}
\end{minipage}
\hspace{1cm}
\begin{minipage}[b]{0.45\linewidth}
\centering
\includegraphics[width=0.75\linewidth]{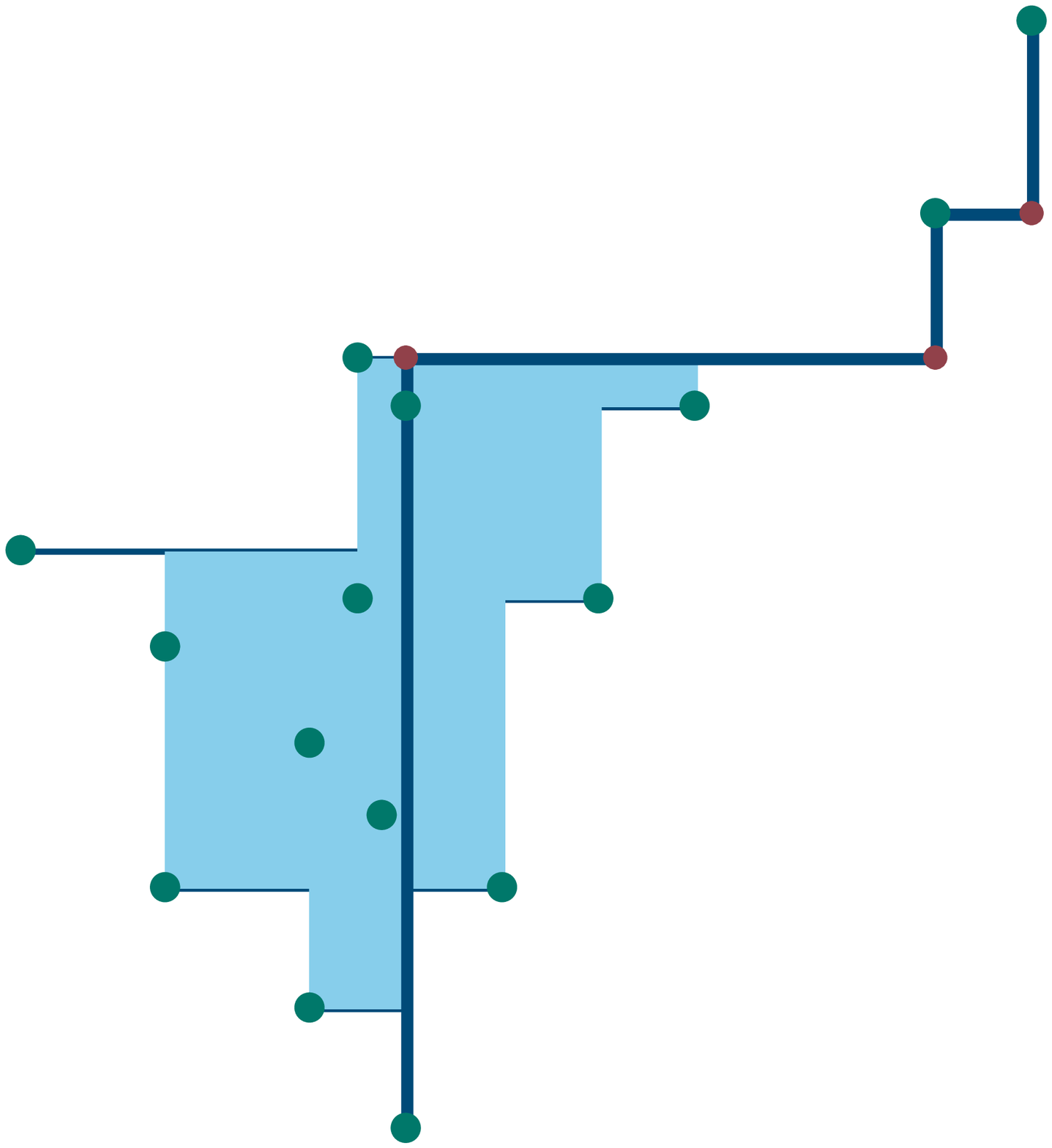}
\end{minipage}
\newline
\begin{minipage}[b]{0.45\linewidth}
\centering
(c) The {\rm SKELETON} of $A$
\newline
\end{minipage}
\hspace{1cm}
\begin{minipage}[b]{0.45\linewidth}
\centering
(d) The vertical hatching of the {\rm SKELETON}
\end{minipage}
\newline\newline
\begin{minipage}[b]{0.90\linewidth}
\centering
\includegraphics[width=0.375\linewidth]{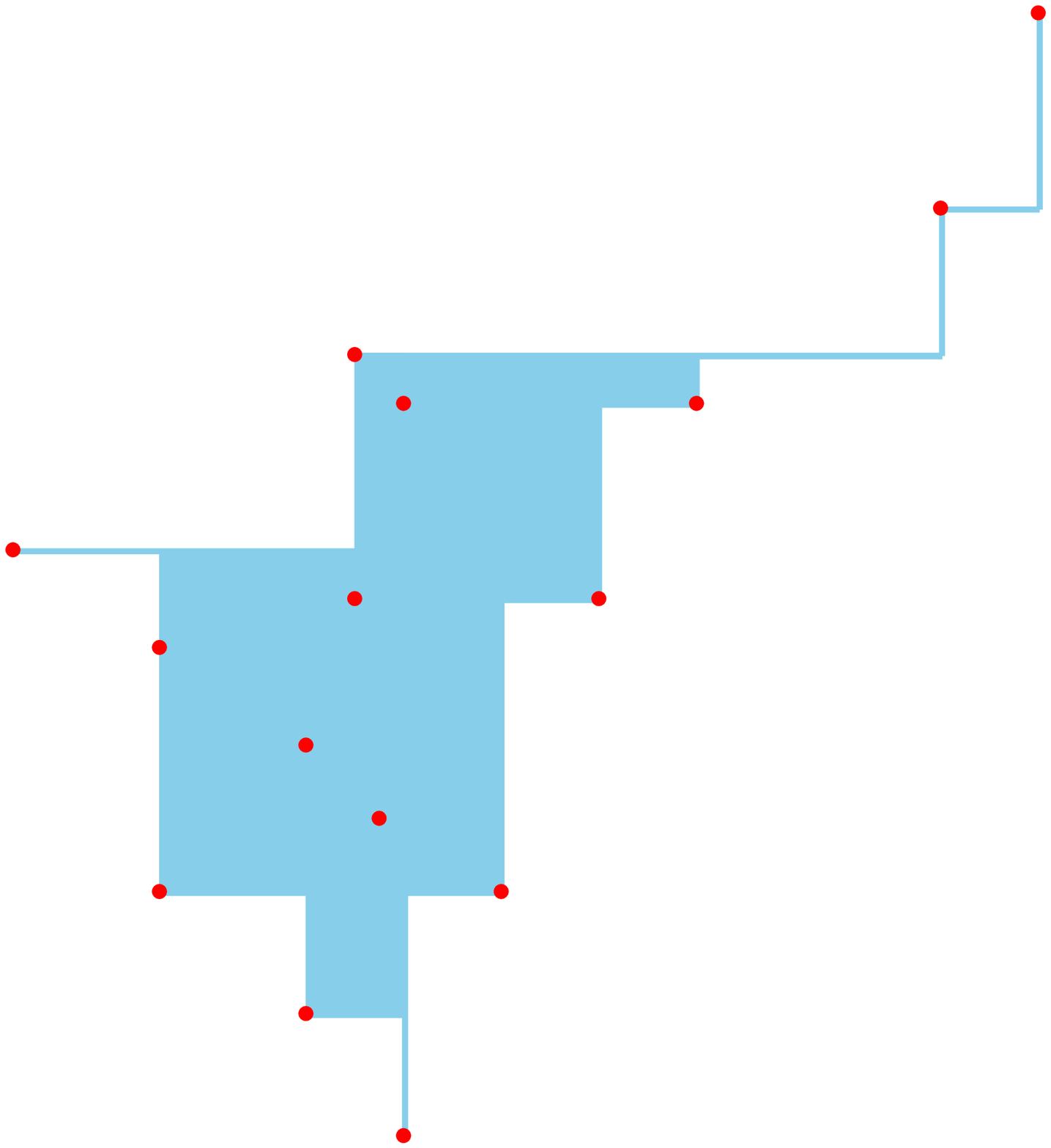}
\end{minipage}
\newline
\begin{minipage}[b]{0.95\linewidth}
\centering
(e) The tight span of the finite set $A$
\end{minipage}
\vspace{0.25cm}
\caption{Construction of the tight span of a $15$-point set given in Figure~\ref{Fig:ed15points}a}\label{Fig:ed15points}
\end{figure}
\end{example}

\clearpage

\section*{Appendix}
We give below the Maple 2015 source-code of the program which generates the tight span of a given finite set step by step according to the algorithm given in Therem~\ref{sn2}.
\bigskip

\begin{maplegroup}
\begin{mapleinput}
\mapleinline{active}{1d}{restart: with(plots): with(plottools):

# SET THE NUMBER OF POINTS OF YOUR FINITE SET AND INPUT THE POINTS OF YOUR FINITE SET (AS (x[i,0],y[i,0])) AND RUN THE PROGRAM.

numberofpoints:=5:
x[1,0]:=0: y[1,0]:=-0.5:
x[2,0]:=-1: y[2,0]:=3:
x[3,0]:=1.5: y[3,0]:=2.5:
x[4,0]:=2: y[4,0]:=0:
x[5,0]:=-0.5: y[5,0]:=0.5:

\bigskip
# DEFINING PROCEDURES
MinOrdinatePoint:=proc(A)
local t, minordinate, xmin, ymin, MinOrdSet;
minordinate:=min(seq(A[i][2],i=1..nops(A))): MinOrdSet:=\{\}:
for t in A do
  if t[2]=minordinate then MinOrdSet:=MinOrdSet union \{t\}: fi;
od;
xmin:=min(seq(MinOrdSet[i][1],i=1..nops(MinOrdSet))): xmin,minordinate;
end proc:
\bigskip
MinApsisPoint:=proc(A)
local t, minapsis, xmin, ymin;
minapsis:=min(seq(A[i][1],i=1..nops(A))):
for t in A do
  if t[1]=minapsis then xmin:=t[1]: ymin:=t[2]: break; fi;
od; xmin,ymin;
end proc:
\bigskip
MaxApsisPoint:=proc(A)
local t, maxapsis, xmax, ymax;
maxapsis:=max(seq(A[i][1],i=1..nops(A))):
for t in A do
  if t[1]=maxapsis then xmax:=t[1]: ymax:=t[2]: break; fi;
od; xmax,ymax;
end proc:
\bigskip
PositionCode:=proc(x0,y0,A)
local left, right,  pc, t;
left:=0: right:=0:
for t in  A minus \{[x0,y0]\} do
if x0>=t[1] and y0<t[2] then left:=1; fi;
od;
for t in  A minus \{[x0,y0]\} do
if x0<=t[1] and y0<t[2] then right:=1; fi;
od;
if left=1 and right=1 then pc:=3:
elif  left=0 and right=1 then pc:=2:
elif  left=1 and right=0 then pc:=1:
else pc:=0:
fi:
pc;
end proc:
}{}
\end{mapleinput}

\mapleresult
\end{maplegroup}

\begin{maplegroup}
\begin{mapleinput}
\mapleinline{active}{1d}{

ULeft:=proc(x0,y0,A)
local a,s,i,t;
i:=0:
for t in  A do
  if x0>=t[1] and y0<t[2] then  i:=i+1: a[i]:=t: fi;
od:
\{seq(a[s],s=1..i)\};
end proc:
\bigskip
URight:=proc(x0,y0,A)
local a,s,i,t;
i:=0:
for t in  A do
  if x0<=t[1] and y0<t[2] then  i:=i+1: a[i]:=t: fi;
od:
\{seq(a[s],s=1..i)\};
end proc:
\bigskip
IsUpperNonempty:=proc(x0,y0,A)
local  t;
if max(seq(t[2]-y0,t in A))>0 then true; else false; fi;
end proc:
\bigskip
DrawSpine:=proc(x0,y0,A)
global x,y; local go, go1, go2, t;
if PositionCode(x0,y0,A)=0 then
  x:=x0; y:=y0;
  pointplot(\{(x0,y0)\},color="DeepBlue",symbol=solidcircle,symbolsize=15):
  break:
elif PositionCode(x0,y0,A)=1 then
  go:=min(seq(x0-t[1],t in A minus \{[x0,y0]\})): x:=x0-go; y:=y0;
  line([x0,y0],[x,y],color="DeepBlue",thickness=4);
elif PositionCode(x0,y0,A)=2 then
  go:=min(seq(t[1]-x0,t in A minus \{[x0,y0]\})): x:=x0+go; y:=y0;
  line([x0,y0],[x,y],color="DeepBlue",thickness=4);
elif PositionCode(x0,y0,A)=3 then
  go1:=max(seq(t[2]-y0,t in ULeft(x0,y0,A))):
  go2:=max(seq(t[2]-y0,t in URight(x0,y0,A))):
  go:=min(go1,go2): x:=x0; y:=y0+go;
  line([x0,y0],[x,y],color="DeepBlue",thickness=4);
fi:
end proc:
\bigskip
DrawSpineNormal:=proc(x0,y0,A)
global x,y; local go, go1, go2, t;
if PositionCode(x0,y0,A)=0 then
  x:=x0; y:=y0;
  pointplot(\{(x0,y0)\},color="SkyBlue",symbol=solidcircle,symbolsize=15):
  break:
elif PositionCode(x0,y0,A)=1 then
  go:=min(seq(x0-t[1],t in A minus \{[x0,y0]\})): x:=x0-go; y:=y0;
  line([x0,y0],[x,y],color="SkyBlue",thickness=2);
elif PositionCode(x0,y0,A)=2 then
  go:=min(seq(t[1]-x0,t in A minus \{[x0,y0]\})): x:=x0+go; y:=y0;
  line([x0,y0],[x,y],color="SkyBlue",thickness=2);
elif PositionCode(x0,y0,A)=3 then
  go1:=max(seq(t[2]-y0,t in ULeft(x0,y0,A))):
  go2:=max(seq(t[2]-y0,t in URight(x0,y0,A))):
  go:=min(go1,go2): x:=x0; y:=y0+go;
  line([x0,y0],[x,y],color="SkyBlue",thickness=2);
fi:
end proc:

}{}
\end{mapleinput}

\mapleresult
\end{maplegroup}

\begin{maplegroup}
\begin{mapleinput}
\mapleinline{active}{1d}{
# MAIN BODY OF THE ALGORITHM

S[0]:=\{seq( [x[i,0],y[i,0]],i=1..numberofpoints)\}:
A:=S[0]:
procnum:=0:
x0:=MinOrdinatePoint(A)[1]:
y0:=MinOrdinatePoint(A)[2]:
B:=\{[x0,y0]\}: Pairs:=\{\}: Newpoints:=\{\}:
\bigskip
# GENERATING THE SPINE
while IsUpperNonempty(x0,y0,A) do
procnum:=procnum+1:
a[procnum]:=DrawSpine(x0,y0,A):
anormal[procnum]:=DrawSpineNormal(x0,y0,A):
Pairs:=\{[[x0,y0],[x,y]]\} union Pairs:
if not [x,y] in S[0] then Newpoints:=\{[x,y]\} union Newpoints: fi:
x0:=x; y0:=y:
B:= \{[x0,y0]\} union B:
TA:=ULeft(x0,y0,A) union URight(x0,y0,A):
if PositionCode(x0,y0,TA)=1 then A:=ULeft(x0,y0,A);
elif PositionCode(x0,y0,TA)=2 then A:=URight(x0,y0,A);
elif PositionCode(x0,y0,TA)=3 then A:=ULeft(x0,y0,A) union URight(x0,y0,A);
else A:=\{[x0,y0]\}:
fi:
end:
procnum0:=procnum:
\bigskip
# CONNECTING THE POINTS TO THE SPINE TO OBTAIN SKELETON
C:=S[0] minus B:
for t in C do
  for i from 1 to nops(B) do
  for j from 1 to nops(B) do
  if t[2]>=B[i][2] and t[2]<=B[j][2] and B[i][1]=B[j][1] then
    procnum:=procnum+1:
    a[procnum]:=line(t,[B[i][1],t[2]],color="DeepBlue",thickness=2);
    anormal[procnum]:=line(t,[B[i][1],t[2]],color="SkyBlue",thickness=2);
    Pairs:=\{[t,[B[i][1],t[2]]]\} union Pairs:
  fi:
  od:od:od:
procnum1:=procnum:
\bigskip
# HATCHING OPERATION TO THE SKELETON
minapsis:=MinApsisPoint(S[0])[1]:
maxapsis:=MaxApsisPoint(S[0])[1]:
stepsize:=0.01:
#You can decrease the stepsize for more dense hatching process
for i from minapsis to maxapsis by stepsize do
  count:=0:
  for t in Pairs do
  if i<t[1][1] and i>t[2][1] then count:=count+1: ord[count]:=t[1][2]:
  elif i>t[1][1] and i<t[2][1] then count:=count+1: ord[count]:=t[1][2]:
  fi:
  od:
  if count>1 then
    ordmin:=min(seq(ord[z],z=1..count)):
    ordmax:=max(seq(ord[z],z=1..count)):
    procnum:=procnum+1:
    a[procnum]:=line([i,ordmin],[i,ordmax],color="SkyBlue",transparency=0.9);
    anormal[procnum]:=line([i,ordmin],[i,ordmax],color="SkyBlue");
  fi:
  od:
}{}
\end{mapleinput}

\mapleresult
\end{maplegroup}

\begin{maplegroup}
\begin{mapleinput}
\mapleinline{active}{1d}{

# NOW IT IS DRAWING THE TIGHT SPAN STEP BY STEP
a[0]:=pointplot(\{seq([S[0][n][1],S[0][n][2]],n=1..numberofpoints)\},
symbol=solidcircle,color="BlueGreen",symbolsize=20):
anormal[0]:=pointplot(\{seq([S[0][n][1],S[0][n][2]],n=1..numberofpoints)\},
symbol=solidcircle,color="Red",symbolsize=10):
newpoint:=pointplot(\{seq([Newpoints[n][1],Newpoints[n][2]],n=1..nops(Newpoints))\},
symbol=solidcircle,color="PaleRed",symbolsize=16):
\bigskip
display(a[0],axes=normal,scaling=constrained);
display(seq(a[i],i=1..procnum0),newpoint,a[0],axes=none,scaling=constrained);
display(seq(a[i],i=procnum0+1..procnum1),seq(a[i],i=1..procnum0),newpoint,a[0],
axes=none,scaling=constrained);
display(seq(a[i],i=procnum0+1..procnum),seq(a[i],i=1..procnum0),newpoint,a[0],
axes=none,scaling=constrained);
display(seq(anormal[i],i=1..procnum),anormal[0],,axes=none,scaling=constrained);

}{}
\end{mapleinput}

\mapleresult
\end{maplegroup}


\bigskip






\begin{thebibliography}{1}
\bibitem{bur} D. Burago, Y. Burago, S. Ivanov, A Course in Metric Geometry, Graduate Studies in Mathematics, American Mathematical Society, USA, 2001.

\bibitem{dre} A. Dress, Trees, tight extensions of metric spaces, and
the cohomological dimension of certain groups: A note on
combinatorial properties of metric spaces, Advances in Mathematics
53 (1984), 321--402.

\bibitem{epp} D. Eppstein, Optimally fast incremental Manhattan plane embedding
and planar tight span construction, Journal of Computational Geometry 2(1) (2011), 144--182.

\bibitem{epp1} D. Eppstein, personal communication.

\bibitem{isb} J. R. Isbell, Six theorems about injective metric spaces, Comment. Math. Helvetici 39 (1964), 65--76.

\bibitem{kilicarxivilk} M. Kilic, S. Kocak, Tight Span of Subsets of The Plane With The Maximum Metric, ArXiv:1506.05982.

\bibitem{kilicarxiv} M. Kilic, S. Kocak, Tight Span of Path Connected Subsets of the Manhattan Plane, ArXiv:1507.05041.

\bibitem{pap} A. Papadopoulos, Metric Spaces, Convexity and Nonpositive Curvature, Irma Lectures in Mathematics and Theoretical Physics, European Mathematical Society, Germany, 2005.

\end{thebibliography}
 \end{document}